\documentclass[a4paper,12pt,reqno]{amsart}

\usepackage[utf8x]{inputenc}
\usepackage{amssymb}
\usepackage{graphicx}
\usepackage{libertine}
\usepackage{mathtools}
\usepackage{pgfplots}
\usepackage{geometry}
\usepackage{scalerel}
\usepackage[colorlinks,allcolors=blue]{hyperref}

\allowdisplaybreaks
\pgfplotsset{compat=1.16}

\newcommand{\p}{\mathbb{P}}
\newcommand{\e}{\mathbb{E}}
\newcommand{\var}{\mathrm{Var}}
\newcommand{\cip}{\stackrel{\p}{\to}}
\newcommand{\cid}{\stackrel{d}{\to}}

\newcommand{\Oh}{\mathrm{O}}
\newcommand{\n}{{(n)}}
\newcommand{\ges}{\geqslant}
\newcommand{\les}{\leqslant}
\newcommand{\ov}{\overline}
\newcommand{\und}{\underline}
\newcommand{\1}[1]{\mbox{\rm\large 1}_{\{#1\}}}
\newcommand\wh[1]{\hstretch{2}{\hat{\hstretch{.5}{#1}}}}

\newcommand{\ve}{\varepsilon}
\newcommand{\mbR}{\mathbb{R}}
\newcommand{\mcC}{\mathcal{C}}
\newcommand{\mcN}{\mathcal{N}}
\newcommand{\mcW}{\mathcal{W}}
\newcommand{\D}{\mathrm{d}}

\DeclarePairedDelimiter{\abs}{\lvert}{\rvert}
\DeclarePairedDelimiter{\RD}{\lfloor}{\rfloor}

\theoremstyle{plain}
\newtheorem{theorem}{Theorem}
\newtheorem{lemma}[theorem]{Lemma}
\newtheorem{proposition}[theorem]{Proposition}
\newtheorem{corollary}[theorem]{Corollary}

\theoremstyle{remark}
\newtheorem{remark}[theorem]{Remark}

\numberwithin{theorem}{section}

\begin{document}
\begin{abstract}
We provide a simple algorithm for construction of Brownian paths approximating those of a L\'evy process on a finite time interval.
It requires knowledge of the L\'evy process trajectory on a chosen regular grid and the law of its endpoint, or the ability to simulate from that.
This algorithm is based on reordering of Brownian increments, and it can be applied in a recursive manner.
We establish an upper bound on the mean squared maximal distance between the paths and determine a suitable mesh size in various asymptotic regimes. The analysis proceeds by reduction to the comonotonic coupling of increments.
Applications to model risk and multilevel Monte Carlo are discussed in detail, and numerical examples are provided.
\end{abstract}

\title{Implementable coupling of L\'evy process and Brownian motion}
\author[V.\ Fomichov, J.\ Gonz\'alez C\'azares and J.\ Ivanovs]{Vladimir Fomichov, Jorge Gonz\'alez C\'azares and Jevgenijs Ivanovs}
\address{Aarhus University, the Alan Turing Institute and University of Warwick}
\maketitle

\section{Introduction}
Let $X=(X(t),t\ges 0)$ be a L\'evy process with $X(1)$ having zero mean and unit variance. We aim to construct a standard Brownian motion $W$ on the same probability space (or its extension) such that the mean squared maximal distance
\begin{equation}
\label{eq:objective}
\e\sup_{t\in [0,1]} \abs{X(t)-W(t)}^2
\end{equation}
is small. This, in particular, provides an upper bound on the respective Wasserstein distance between the laws of the given L\'evy process and the standard Brownian motion on the time interval $[0,1]$, assuming the above expectation can be computed with some guaranteed accuracy. Such bounds are needed in a number of applications including model risk and distributionally robust optimization.
Our focus, however, is on simple explicit constructions allowing to efficiently generate Brownian trajectories, which are paramount, for example, in the multilevel Monte Carlo method discussed below.
It should be mentioned that our problem is different from so-called Markov couplings of L\'evy processes~\cite{schilling11,majka17}, where the two processes of interest have the same transition probabilities, but different initial distributions.

\subsection{Applications}
\subsubsection*{Model Risk and Distributionally Robust Optimization}
Every model is only an approximation of reality and thus it is important to understand the impact of model misspecification (going beyond model parameters) on the quantities of interest.
A popular approach is to find the worst-case values for a family of plausible models in some neighbourhood of a chosen baseline model, which can be viewed as a systematic stress test~\cite{breuer2013systematic}.
The Wasserstein distance is a natural choice to define such an ambiguity ball of models in various settings, and it often leads to simple worst-case expressions~\cite{blanchet19}.
Furthermore, this approach has close links to distributionally robust optimization, where allowing for some freedom in the model helps to mitigate the optimizer's curse~\cite{esfahani2018data}, which is a well-known phenomenon in stochastic programming.

A critical task in these procedures is to estimate the radius of the ambiguity ball, and in the case of L\'evy-driven models our coupling provides the necessary tool.
For example, in the context of ruin theory one may use a diffusion approximation as a baseline model, while insisting that a certain compound Poisson process
(corresponding to the classical Cram\'er--Lundberg risk model) belongs to the ambiguity ball. 
In~\cite{blanchet19} it has been suggested to simulate coupled paths (using the Brownian embedding in~\cite{khoshnevisan93}) and to choose the radius according to the square-root of the empirical counterpart of~\eqref{eq:objective}.
In this setting an upper bound on the Wasserstein distance between the processes would be sufficient, given it is not overly conservative.
Note, however, that the implementable coupling approach can be based on the available historical claim sizes (jumps) without the knowledge of the exact underlying distribution.

\subsubsection*{Multilevel Monte Carlo}
Suppose we want to estimate the value of the expectation 
$\e g(X)$ for some appropriate function $g$ of the path of~$X$ using Monte Carlo simulation. Sampling $g(X)$ exactly is not feasible in general.
One standard way around this problem is to fix a small truncation level $\ve_n>0$ and to consider an approximation $X_n$ of $X$ obtained by replacing the martingale of jumps in $[-\ve_n,\ve_n]$ by an appropriately scaled Brownian motion~\cite{MR1834755}.
Apart from estimating $\e g(X_n)$ one also needs to control the bias $\e[g(X_n)-g(X)]$. 
A significant improvement in the computational complexity can be often obtained by the multilevel Monte Carlo method~\cite{giles2015multilevel}, see also~\cite{kyprianou_MLMC,giles2017multilevel_levy} for the L\'evy process setting.
This, however, requires the ability to sample jointly $g(X_n)$ and $g(X_{n+1})$ in a way that the level variance $\var[g(X_{n+1})-g(X_n)]$ is as small 
as possible, which must be done without significantly increasing the cost of drawing 
such a joint sample compared with the single level sample.

It is clear from the construction of approximations  that the only difficult part is to couple the martingale of jumps in $[-\ve_n,-\ve_{n+1})\cup (\ve_{n+1},\ve_n]$, where $\ve_{n+1}<\ve_n$, with an appropriately scaled Brownian motion, since the other parts can be reused. Note that such a Brownian motion must be constructed and it is not sufficient to simply claim its existence.
Assuming $g$ is sufficiently 
regular, it typically suffices to control the mean squared maximal distance between the paths as in~\eqref{eq:objective}.
Our coupling provides the necessary tool, and the corresponding asymptotic analysis is presented in \S\ref{sec:multilevel}.
In particular, we show that a substantial gain in computational complexity can be achieved for processes $X$ of unbounded variation on compacts over the algorithms available in the existing literature (see \S\ref{subsec:discr} and Figure~\ref{fig:comparison} below for a comparison with~\cite{MR2759203,dereich_heidenreich_11}).

\subsection{The construction}
\label{sec:construction}
Here we describe in words our construction of the coupled 
Brownian motion~$W(t)$ on the time interval $[0,1]$, 
postponing its precise definition and related notation to 
\S\ref{sec:methods}. This construction depends on an 
integer $k\ges 1$, and it can be viewed in a number of alternative 
ways. Firstly, we construct a skeleton of a standard 
Brownian motion~$W'(i/k)$ for $i=0,\ldots,k$ such that $W'(1)$ 
is (nearly) comonotonically coupled with $X(1)$, but is otherwise 
independent. The law of $X(1)$ may not be explicit, in which case  
we may use either fast Fourier inversion~\cite{MR1620156} or 
an empirical counterpart that only requires the ability to 
simulate $X$ at time~$1$ (see Appendix~\ref{sec:final_value}). 
Secondly, we reorder the $k$ increments 
$\Delta^k_i W'=W'(i/k)-W'((i-1)/k)$ so that they match the 
ordering of the increments $\Delta^k_i X$ of the process $X$; the ties in the latter are resolved randomly. 
The accumulated reordered increments form the skeleton of 
another Brownian motion $W$ over the grid $0,1/k,\ldots,1$, 
and we supplement this skeleton with independently sampled 
Brownian bridges connecting these points, see Figure~\ref{fig:OrderIncrements} for an illustration.

\begin{remark}\label{rem:match_ECDF}
The reordering of the increments matches their ranks, so a random draw of any pair of increments is equivalent to a draw from the comonotonically coupled empirical measures based on the sampled increments. In other words, our reordering procedure exploits the convergence of empirical measures to their true distributions and the simple form of the comonotonic coupling between discrete distributions with uniform weights. 
\end{remark}

\begin{figure}[ht]
\centering
\includegraphics[width=1\textwidth]{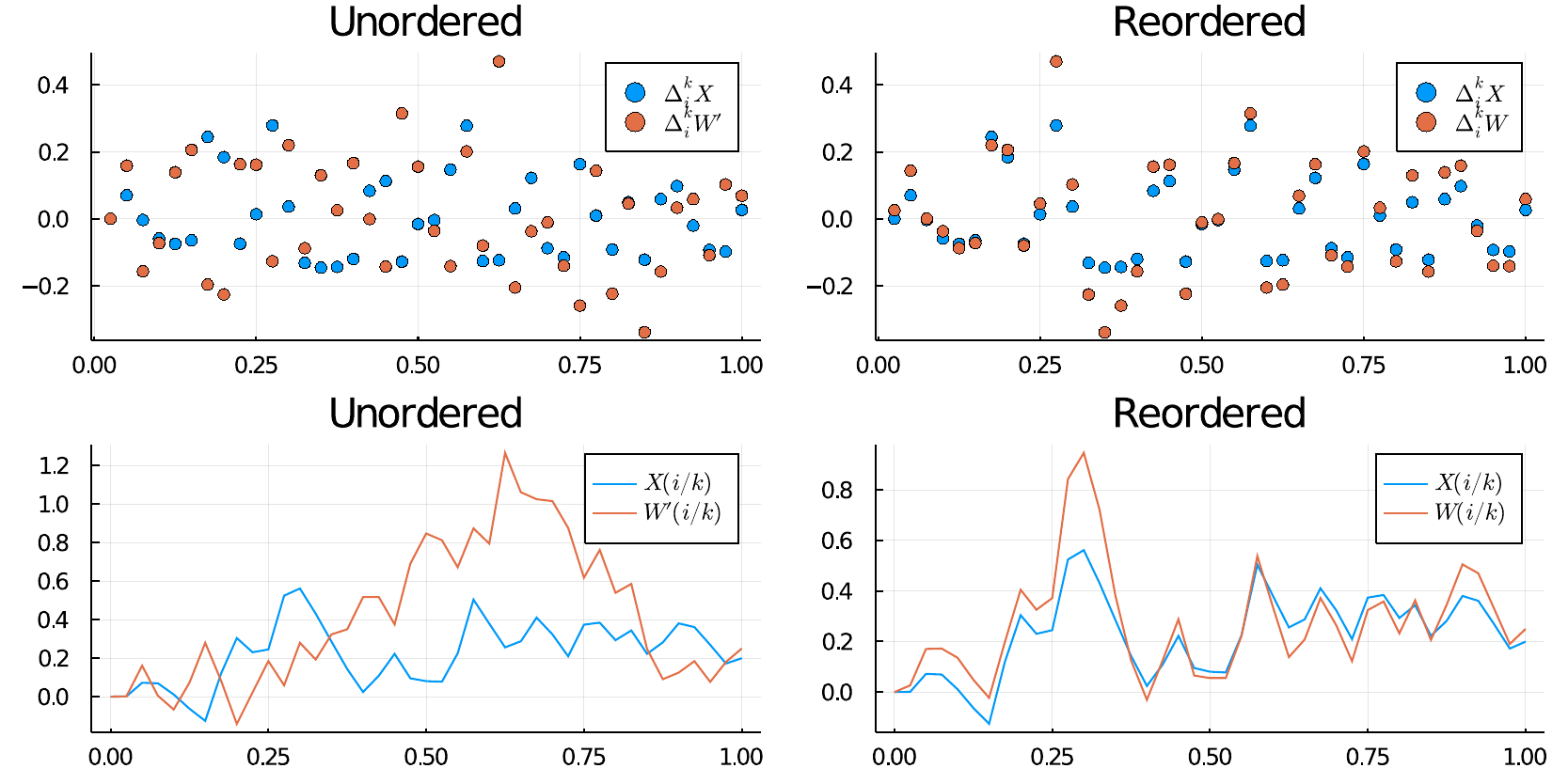}
\caption{Increments (top) and skeletons (bottom) of a gamma process $X$ and Brownian motions $W',W$ over the uniform grid $\{0,1/k,\ldots,1\}$ for $k=40$.}
\label{fig:OrderIncrements}
\end{figure}

It is clear that the quality of our coupling strongly depends on the choice of $k$. For $k=1$ only the end-points are coupled, but the Brownian bridge  is independent of~$X$.
The case of a large $k$ can be understood using the results in~\cite{recoverBM}, where we employed the idea of reordering independently sampled normal increments to recover the Brownian component of~$X$ up to a random linear drift.
More precisely, letting $W_X$ be the Brownian component of $X$ normalized to have unit variance, 
the Brownian bridge $W(t)-tW(1)$ converges in probability in supremum norm to the bridge $W_X(t)-tW_X(1)$ as $k\to\infty$.
In the case when $X$ has no Brownian component (under a further minor condition) the weak limit of the pair of processes $(W(t)-tW(1),X(t))$  has independent components.
In such a case $k=1$ and $k=\infty$ are analogous in the sense that they both lead to an independent Brownian bridge.
Therefore, one needs to choose a moderately large~$k$.

Our main result, Corollary~\ref{cor:Levy_to_BM}, provides some asymptotic theory suggesting an adequate choice of~$k$.
We further investigate it in~\S\ref{sec:regimes} for various limiting regimes, apply it to multilevel Monte Carlo method in~\S\ref{sec:multilevel}, and then also illustrate various choices numerically in~\S\ref{sec:numerical_simulations}.

\begin{remark}\label{rem:hierarchical}
Note that the coupled skeleton of $W$ is supplemented with independent Brownian bridges.
Alternatively, one may repeat the partitioning and reordering procedure for each interval $[(i-1)/k,i/k]$ hoping to construct bridges which better approximate the evolution of $X$ within these intervals.
This leads to a hierarchical construction with possibly different $k$ at each level. In \S\ref{sec:numerical_simulations} we examine this idea numerically.
\end{remark}

\subsection{Literature and related methods}
Coupling of random variables and processes is a classical area of probability theory, see monographs~\cite{thorisson,rio} and references therein. 
The famous result of Skorokhod in its modern interpretation~\cite[Thm.~15.17]{kallenberg} states that given a sequence of L\'evy processes $X_n$ with $X_n(1)\cid W(1)$ it is possible to construct $X_n$ and a standard Brownian motion $W$ on a common probability space so that $\sup_{t\in [0,1]} \abs{X_n(t)-W(t)}\cip 0$. This existence result relies on Skorokhod's representation theorem~\cite[Thm.~4.30]{kallenberg}, and it does not yield a feasible algorithm for constructing such couplings.

Another related classical result is Strassen's random walk approximation by a Brownian motion~\cite{strassen1964invariance} underlying the functional 
LIL and based on Skorokhod's embedding theorem, see also~\cite[Ch.~14]{kallenberg}. 
It yields an explicit coupling, implementation of which requires sampling Brownian paths up to a certain stopping time. It is, however, not obvious how to do this both accurately and efficiently.
An embedding of a given L\'evy process into a Brownian motion by means of a random time change was used in~\cite{monroe} to investigate its $p$-variation.
Furthermore, a construction of a compensated Poisson process $X$ from a Brownian motion $W$ was suggested in~\cite{khoshnevisan93} and then used to study crossings of empirical processes. The analogous construction is also suggested for a compound Poisson process with negative jumps, but the resultant process in general exhibits positive dependence between inter-arrival times and subsequent jumps, and so it is not compound Poisson.

Yet another coupling of random walks with exponential moments, often called the Hungarian embedding or the KMT (Koml\'os--Major--Tusn\'ady) coupling, was proposed in~\cite{komlos_major_tusnady_75}; it was later extended by Zaitsev to the multidimensional case in~\cite{zaitsev}. Although, in a certain sense, these couplings are optimal, the fact that they are based on conditional distributions makes them hardly suitable for numerical computations. The same critique applies to the coupling of solutions of two SDEs with respect to a L\'evy process and a Brownian motion constructed in~\cite{fournier} and of a L\'evy process with small jumps and a Brownian motion introduced in~\cite{MR2759203}, both of which are based on the KMT coupling and its multidimensional extension by Zaitsev.

A somewhat related problem consists in estimating the Wasserstein distance between two infinitely divisible distributions, or equivalently between the marginal distributions of two L\'evy processes.
The optimal, or near-optimal, bounds in the general case were recently obtained in~\cite{mariucci_reiss} relying on the corresponding bound in~\cite{rio09} on the error in the central limit theorem, see also Lemma~\ref{lem:Rio} below.

\section{Definitions and fundamental properties}
\label{sec:methods}
We start by setting up the notation and stating a result about the quality of comonotonic coupling of the marginals.
Next, we introduce an auxiliary coupling method based on comonotonic coupling of $k$ increments, show that the respective Brownian bridges are closely related, and investigate associate discretization errors.
These notions will be used in \S\ref{sec:main_results} to show (under mild conditions) that our reordering coupling is not worse than the above auxiliary coupling.
Finally, the quality of the latter is analysed to arrive at our main result in Corollary~\ref{cor:Levy_to_BM}.
Unlike in \S\ref{sec:main_results}, where we work with a convergent sequence $X_n$ of L\'evy processes, here we deal with a single process $X$ which greatly simplifies presentation. 

Throughout the paper, unless explicitly stated otherwise, we assume that
\begin{equation}
\label{eq:main_assumptions}
\e X(1)=0,\quad \e X^2(1)=1,\quad \e X^4(1)<\infty.
\end{equation}
The L\'evy measure of $X$ will be denoted by $\Pi$, and we also define its fourth moment and the tail function:
\[
\mu_4\coloneqq\int_\mbR x^4\Pi(\D x),\qquad \ov\Pi(x)\coloneqq\Pi(\mbR\setminus [-x,x]),x>0.
\]
Recall that the marginal distribution functions are continuous unless $X$ is a drifted compound Poisson process, and so this fundamental class of processes should be treated with care.

\subsection{Wasserstein distance between the marginal laws}
Consider two random variables $\zeta_1$ and $\zeta_2$ with distribution functions $F_1$ and $F_2$, respectively.
Recall that every coupling $(\zeta_1,\zeta_2)$ can be obtained by setting $\zeta_2=h(\zeta_1,U)$ for an appropriate measurable function~$h$ and a standard uniform~$U$ independent of $\zeta_1$.
That is, every joint law with marginals $F_1$ and $F_2$ can be retrieved in this way. Comonotonic coupling is usually defined by $\big(F_1^{-1}(U),F_2^{-1}(U)\big)$,
where $F_i^{-1}(u)=\inf\{x|F_i(x)\ges u\}$ denotes the left inverse of~$F_i$.
Alternatively, it can be obtained by taking
\begin{equation}
\label{eq:h}
h(x,u)=F_2^{-1}\big(\p(\zeta_1<x)+u\p(\zeta_1=x)\big)
\end{equation}
in the specification of $\zeta_2$, see~\cite{ruschendorf2009distributional}. 
For a continuous $F_1$ there is no need to further randomize using $U$, and we see that $\zeta_2=F_2^{-1}(F_1(\zeta_1))$ is a monotone transform of~$\zeta_1$.
The Wasserstein distance $\mcW_2$ between the laws of $\zeta_1$ and $\zeta_2$ (with some abuse of notation) is defined as
\[
\mcW_2(\zeta_1,\zeta_2)\coloneqq \inf_{(\zeta_1',\zeta_2')}\left(\e\abs{\zeta'_1-\zeta'_2}^2\right)^{1/2},
\]
where the infimum is taken over all possible couplings of $\zeta_1$ and $\zeta_2$.
It is a standard fact that the infimum is achieved by the comonotonic coupling~\cite[Ex.~3.2.14]{mass_transportation}.

The following lemma provides an upper bound for the Wasserstein distance between the marginals of a L\'evy process and a Brownian motion in terms of $\mu_4$, see also~\cite{fournier, mariucci_reiss}.

\begin{lemma}
\label{lem:Rio}
There exists a constant $C>0$ such that for any L\'evy process $X$ satisfying condition~\eqref{eq:main_assumptions} and a standard Brownian motion $B$ we have
\[
\forall\, t\ges 0:\quad \mcW_2^2(X(t),B(t))\les C\mu_4.
\]
\end{lemma}

\begin{proof}
Obviously, we can assume that $t>0$. For arbitrary $n\ges 1$ we have
\[
X(t)=\sum_{i=1}^n [X(it/n)-X((i-1)t/n)],
\]
where the summands are i.i.d., and so applying~\cite[Thm.~4.1]{rio09} yields
\[
\mcW_2^2(X(t),B(t))=t\mcW_2^2(X(t)/\sqrt{t},B(1))\les tCn\e\big(\abs{X(t/n)}/\sqrt{t}\big)^4=C(n/t)\e X^4(t/n).
\]
By~\cite[Thm.~1.1]{MR2479503}, the upper bound tends to the upper bound of the stated inequality as $n\to\infty$.
\end{proof}


\subsection{Two coupling methods}\label{sec:twomethods}
Generally speaking, coupling marginal distributions of 
two processes does not allow to construct a coupling of the 
whole processes. 
Nevertheless, we may comonotonically couple the increments of a Brownian motion and $X$ over a grid.
This auxiliary method is closely related to the proposed coupling based on reordering of increments, and it will be important for the asymptotic analysis of the latter.
Note that the auxiliary coupling is much more challenging to implement efficiently in a common scenario when the distribution function of the increments of $X$ is not readily available. 
Here we provide precise definitions of the two methods, both of which depend on
\[\text{integer } k\ges 1\; \text{specifying the number of increments used.}\]
We will need standard uniforms $U,U_1,\ldots, U_k$ and also $k+1$ Brownian bridges, which are mutually independent and also independent of~$X$.

\subsubsection*{Reordering of increments}
Let us now give a precise definition of the process $W$ described in words in \S\ref{sec:construction}. 
Start by taking a standard Brownian motion $W'$ with $W'(1)=h(X(1),U)$ for an appropriate measurable function $h$ and such that $(W'(t)-tW'(1),t\in[0,1])$ is a standard Brownian bridge independent of everything else.
Recall that this construction allows for an arbitrary coupling of the end-points $X(1)$ and $W'(1)$.
Importantly, the processes $W'$ and $X$ are independent given $X(1)$.

Let $\pi$ be an a.s.\ unique random 
permutation on $\{1,\ldots,k\}$ such that for all $i\neq j$:
\begin{equation}
\label{eq:permutation}
\Delta_{\pi(i)}^kW'<\Delta_{\pi(j)}^kW'\qquad\text{iff}
\qquad
\Delta_i^kX<\Delta_j^kX
\qquad\text{or}\qquad
\Delta_i^kX=\Delta_j^kX,\enskip U_i<U_j.
\end{equation}
That is, the ties in $\Delta_i^k X$, only possible when $X$ is a compound Poisson process with drift, are resolved uniformly at random. Define a stochastic process $W$ by setting $W(0)\coloneqq 0$ and
\begin{equation}\label{eq:W}
W(t)\coloneqq W\big(\tfrac{i-1}{k}\big)+W'\big(\tfrac{\pi(i)-1}{k}+t-\tfrac{i-1}{k}\big)- W'\big(\tfrac{\pi(i)-1}{k}\big),\quad \tfrac{i-1}{k}<t\les\tfrac{i}{k},\quad i=1,\ldots,k.
\end{equation}
Thus defined $W$ is indeed a standard Brownian motion, see Lemma~\ref{lem:BM} below.

\subsubsection*{Comonotonic coupling of increments}
Let $h_k$ be the function defined in~\eqref{eq:h}, where $\zeta_1=X(1/k)$ and $\zeta_2$ is mean zero normal with variance~$1/k$.
By taking $\xi_i=h_k(\Delta_i^kX,U_i)$ we produce Brownian increments comonotonically coupled with the increments of~$X$.
Let also $\beta_1,\ldots,\beta_k$ be independent Brownian bridges on the time interval $[0,1/k]$ that are independent of everything else. 
Define a stochastic process $\wh W$ by setting $\wh W(0)\coloneqq 0$ and
\begin{equation}\label{eq:W0}
\wh W(t)\coloneqq \wh W\left(\tfrac{i-1}{k}\right)+\beta_i\left(t-\tfrac{i-1}{k}\right)+ k\left(t-\tfrac{i-1}{k}\right)\xi_i,\quad \tfrac{i-1}{k}<t\les\tfrac{i}{k},\quad i=1,\ldots,k.
\end{equation}
In words, $\wh W$ is a stochastic process whose increments over the given grid are comonotonically coupled with the increments of $X$ and whose ``interval bridges'' are independent of $X$.

\begin{lemma}\label{lem:BM}
The processes $W$ and $\wh W$ are standard Brownian motions and their increments have the same ordering:
\[\Delta_i^kW<\Delta_j^k W\qquad\text{iff}\qquad \Delta_i^k\wh W<\Delta_j^k \wh W\]
with probability~1.
\end{lemma}

\begin{proof}
The same ordering of increments follows from~\eqref{eq:permutation} and the fact that $h_k(x_1,u_1)<h_k(x_2,u_2)$ when either $x_1<x_2$ or $x_1=x_2$ and $u_1<u_2$ with $\p(X(1/k)=x_1)>0$; here $u_i\in(0,1)$.
It is essential that the same sequence of $U_i$ was used to resolve the ties in the first method and to provide extra randomness in the second method.

It is easy to see that $\wh W$ is a standard Brownian motion, since the increments $\xi_i$ are independent and have the required distribution.
Checking that $W$ is a Brownian motion is more complicated, and we need to take care of both reordering and coupling the end-points.
Recall that the process $X$, Brownian bridge $W'(t)-tW'(1)$, and uniforms~$U,U_1,\ldots, U_k$ are mutually independent.
Conditionally on $X(1)$, the pairs $(\Delta_i^k X,U_i)$ are exchangeable, and thus $W(t)-tW(1)$ is still a standard Brownian bridge.
The latter can be seen as obtained from  $W'(t)-tW'(1)$ by an independent uniform permutation of increments.
Note that $W(1)=W'(1)=h(X(1),U)$ to conclude.
\end{proof}

\begin{remark}
If $X$ is not a compound Poisson process with drift, then there is no need in the uniforms $U_1,\ldots,U_k$ in the above two constructions.
In such a case there are no ties in the first method, the increments $\xi_i$ are monotone transforms of $\Delta_i^k X$ in the second method, and the same ordering of the increments of $W$ and $\wh W$ is automatic.
\end{remark}

Finally, let us point out that we have constructed  a trivariate process $(X,W,\wh W)$, and not just two couplings $(X,W)$ and $(X,\wh W)$.
This is crucial for the proofs, where the quality of the reordering coupling is related to the quality of the auxiliary coupling, see Theorem~\ref{thm:basic} below.

\subsection{Proximity of the two Brownian bridges}
The following lemma is an extension of Lemma~7 from~\cite{recoverBM}, which concerned convergence in probability, to convergence in mean.

\begin{lemma}
\label{lem:standard_normal_rvs}
Let $Z_1,\ldots,Z_k$ and $Z'_1,\ldots,Z'_k$ be two independent sets of i.i.d. standard normal random variables, and let $Z_{(1)},\ldots,Z_{(k)}$ and $Z'_{(1)},\ldots,Z'_{(k)}$ be their order statistics. Then
\begin{gather*}
\lim_{k\to\infty} \dfrac{1}{\log\log k}\e\sum_{i=1}^k \big((Z_{(i)}-\ov Z_k)-(Z'_{(i)}-\ov Z'_k)\big)^2=2,
\end{gather*}
where $\ov Z_k=(Z_1+\ldots+Z_k)/k$ and $\ov Z'_k$ denote the respective arithmetic means.
\end{lemma}

\begin{proof}
Let $G_k$ and $G'_k$ be the empirical distribution functions of $Z_1,\ldots,Z_k$ and $Z'_1,\ldots,Z'_k$, respectively. We have
\[
\dfrac{1}{\log\log k}\sum_{i=1}^k \big(Z_{(i)}-Z'_{(i)}\big)^2=\dfrac{k}{\log\log k} \int_0^1 (G_k^{-1}(x)-G_k'^{-1}(x))^2\D x=\dfrac{k\mcW_2^2(G_k,G'_k)}{\log\log k},
\]
where the expectation of the last term converges to $2$ by~\cite[Rem.~2.4]{MR4059190}. Also, note that
\[
\sum_{i=1}^k \big((Z_{(i)}-\ov Z_k)-(Z'_{(i)}-\ov Z'_k)\big)^2=\sum_{i=1}^k \big(Z_{(i)}-Z'_{(i)}\big)^2-k\big(\ov Z_k-\ov Z'_k\big)^2,
\]
where the difference $\ov Z_k-\ov Z'_k$ has distribution $\mcN(0,2/k)$.
\end{proof}

We are now ready to upper bound the mean squared maximal distance between the bridges corresponding to $W$ and $\wh W$ over the grid with mesh size~$1/k$.
Note that both processes depend on~$k$ as well.

\begin{lemma}
\label{lem:Brownian_motions}
For the Brownian motions $W$ and $\wh W$ defined in~\eqref{eq:W} and~\eqref{eq:W0}  using $k\ges 1$ increments it holds that
\[
\e\max_{1\les i\les k} \abs*{[W(\tfrac{i}{k})-\tfrac{i}{k}W(1)]-[\wh W(\tfrac{i}{k})-\tfrac{i}{k}\wh W(1)]}^2
=\Oh(\log\log k/k),\quad k\to\infty,
\]
uniformly for all processes $X$.
\end{lemma}

\begin{proof}
Observe that the bridge $W(t)-tW(1)$ is obtained by sampling the bridge $W'(t)-tW'(1)$ independently and then reordering its increments (together with in-between evolutions) according to the increments of the bridge $\wh W(t)-t\wh W(1)$, see Lemma~\ref{lem:BM} in particular. Here we rely on the obvious fact that adding a linear trend has no influence on the ordering of increments.
Denoting the increments by
\[
\wh\eta_i=\Delta_i^k\wh W-\Delta_i^kt\cdot\wh W(1),\qquad \eta_i=\Delta_i^kW-\Delta_i^kt\cdot W(1),
\]
observe, using the notation of Lemma~\ref{lem:standard_normal_rvs}, that the vectors $(\sqrt{k}\wh\eta_i)$ and $(\sqrt{k}\eta_i)$ jointly have the law of  $(Z_i-\ov Z_k)$ and $(Z'_i-\ov Z'_k)$ with the latter reordered according to the former.
Since the summation order is arbitrary, we find using Lemma~\ref{lem:standard_normal_rvs} that
\begin{equation}\label{eq:bridges}
\e\sum_{i=1}^k \big(\sqrt{k}\wh\eta_i-\sqrt{k}\eta_i\big)^2=\e\sum_{i=1}^k \big((Z_{(i)}-\ov Z_k)-(Z'_{(i)}-\ov Z'_k)\big)^2=\Oh(\log\log k).
\end{equation}
Note that this quantity does not depend on the underlying process~$X$.

Furthermore, the pairs $(\wh\eta_i,\eta_i)$ are exchangeable and
\[
(\wh\eta_1-\eta_1)+\ldots+(\wh\eta_k-\eta_k)=(\wh\eta_1+\ldots+\wh\eta_k)- (\eta_1+\ldots+\eta_k)=0.
\]
We are going to use Garsia's inequality (see~\cite[p.~26]{chobanyan_levental_salehi} and references therein), which states that for any $p\ges 1$ and any $x_1,\ldots,x_k\in\mbR$ such that $x_1+\ldots+x_k=0$ we have
\[
\dfrac{1}{k!}\sum_{\sigma\in S_k} \max_{1\les i\les k} \abs{x_{\sigma(1)}+\ldots+x_{\sigma(i)}}^p\les C_p\left(\sum_{i=1}^k x_i^2\right)^{p/2},
\]
where the sum is taken over all permutations $\sigma$ of order $k$ and the constant $C_p>0$ depends only on $p$. We apply this inequality to $p=2$ and $x_i=\wh\eta_i-\eta_i$, and take expectations of both sides. Owing to the exchangeability, the expectations of all maximums are equal. Therefore, we obtain
\[
\e\max_{1\les i\les k} \abs*{(\wh\eta_1+\cdots+\wh\eta_i)- (\eta_1+\cdots+\eta_i)}^2\les C_2\e\sum_{i=1}^k (\wh\eta_i-\eta_i)^2,
\]
which is $\Oh(\log\log k/k)$ according to~\eqref{eq:bridges}.
\end{proof}

\subsection{Discretization error}
Finally, we need appropriate bounds on the mean-squared-maximal discretization error of the involved processes.
For an integer $k\ges 1$ let $h^{[k]}(t)=h(\RD{kt}/k)$ be a discretization of the function $h$ over the grid $\{0,1/k,2/k,\ldots,1\}$.
It is a basic fact that a standard Brownian motion $B$ satisfies
\begin{equation}
\label{eq:BM_disc}
\e\sup_{t\in [0,1]} \big(B(t)-B^{[k]}(t)\big)^2=\Oh(\log k/k),\quad k\to\infty,
\end{equation}
see~\cite[Lem.~4.4]{petterson_95}. This also can be seen using extreme value theory and the scaling property of~$B$.
We are now ready to extend the result of Lemma~\ref{lem:Brownian_motions} to the whole interval~$[0,1]$.

\begin{proposition}\label{prop:bridges}
For the Brownian motions $W$ and $\wh W$ defined in~\eqref{eq:W} and~\eqref{eq:W0}  using $k\ges 1$ increments it holds that
\[
\e\sup_{t\in[0,1]} \abs*{[W(t)-tW(1)]-[\wh W(t)-t\wh W(1)]}^2=\Oh(\log k/k)
\]
as  $k\to\infty$, uniformly for all processes $X$.
\end{proposition}

\begin{proof}
Observe the following upper bound:
\[\begin{multlined}
\sup_{t\in[0,1]} \abs*{[W(t)-tW(1)]-[\wh W(t)-t\wh W(1)]}
\les\max_{1\les i\les k} \abs*{[W(\tfrac{i}{k})-\tfrac{i}{k}W(1)]-[\wh W(\tfrac{i}{k})- \tfrac{i}{k}\wh W(1)]}\\
+\sup_{t\in [0,1]} \abs{W(t)-W^{[k]}(t)}+\tfrac{1}{k}\abs{W(1)}+\sup_{t\in [0,1]} \abs{\wh W(t)-\wh W^{[k]}(t)}+\tfrac{1}{k}\abs{\wh W(1)},
\end{multlined}\]
where the second line is an upper bound on the discretization of the two bridges.
It is sufficient to bound the second moments of all the terms, which is achieved by Lemma~\ref{lem:Brownian_motions} and~\eqref{eq:BM_disc}.
\end{proof}

With respect to the discretization of the process~$X$ we note that $\sup_{t\in[0,1]}\abs{X(t)-X^{[k]}(t)}$ converges to the size of the largest jump a.s., which has been the motivation for studying the integrated discretization error in~\cite{jacod_memin}. Hence the mean-squared-maximal error can be small only when $X$ is close to a Brownian motion in some sense. 
We need the following bound in terms of $\mu_4$, see also Theorem~\ref{thm:X-B} below, even though various alternative bounds may be easier to establish.
The quantity $\mu_4$ is normally small when $X$ is close to a Brownian motion, see Lemma~\ref{lem:mu4} for a precise result.

\begin{proposition}
\label{prop:Levy_disc}
There is a constant $C>0$ such that for any L{\'e}vy process $X$ satisfying~\eqref{eq:main_assumptions} and for any $k\ges 2$ the following bound holds:
\[
\e\sup_{t\in [0,1]} \big(X(t)-X^{[k]}(t)\big)^2\les C(k\mu_4+\log k/k).
\]
\end{proposition}

The proof is rather lengthy so we postpone it to Appendix~\ref{app:bound}.

\section{Asymptotic quality of coupling methods}
\label{sec:main_results}
\subsection{Main results}
Given a sequence of L\'evy processes $X_n$ and integers $k_n\ges 1$, let $W_n$ and $\wh W_n$ be the Brownian motions constructed using, respectively, the reordering method and comonotonic coupling of increments with the common discretization parameter~$k_n$. Note that the coupling of $X(1)$ and $W(1)$ is arbitrary.
We start by showing that the two couplings have asymptotically equivalent quality given that the respective coupling of the end-points is sufficiently good.

\begin{theorem}
\label{thm:basic}
Consider a sequence of L\'evy processes $X_n$ and a sequence of integers $k_n\to\infty$. If $\ve_n\downarrow 0$ is such that $\ve_nk_n/\log k_n$ is bounded away from~$0$ and
\[
\e\abs{X_n(1)-W_n(1)}^2=\Oh(\ve_n),\qquad
\e\sup_{t\in [0,1]} \abs{X_n(t)-\wh W_n(t)}^2=\Oh(\ve_n),
\]
then also
\[
\e\sup_{t\in [0,1]} \abs{X_n(t)-W_n(t)}^2=\Oh(\ve_n).
\]
This result is also true when $W$ and $\wh W$ are swapped.
\end{theorem}

\begin{proof}
Start with an obvious bound
\[
\e\sup_{t\in [0,1]} \abs{X_n(t)-W_n(t)}^2\les 2\e\sup_{t\in [0,1]} \abs{X_n(t)-\wh W_n(t)}^2+2\e\sup_{t\in [0,1]}\; \abs{W_n(t)-\wh W_n(t)}^2.
\]
In view of the assumptions, it is only required to consider the second term
\begin{equation*}
\begin{gathered}
\e\sup_{t\in [0,1]} \abs{W_n(t)-\wh W_n(t)}^2\les 2\e\sup_{t\in [0,1]} \abs{[W_n(t)-tW_n(1)]-[\wh W_n(t)-t\wh W_n(1)]}^2\\
\quad+2\e\abs{W_n(1)-\wh W_n(1)}^2.
\end{gathered}
\end{equation*}
According to Proposition~\ref{prop:bridges} we find that the first term on the right hand side is $\Oh(\log k_n/k_n)=\Oh(\ve_n)$.
Finally,
\[
\e\abs{W_n(1)-\wh W_n(1)}^2\les 2\e\abs{X_n(1)-W_n(1)}^2+2\e\abs{X_n(1)-\wh W_n(1)}^2= \Oh(\ve_n),
\]
by assumption. The final statement is proven analogously.
\end{proof}

Observe that, for comonotonically coupled end-points $W_n(1)$ and $X_n(1)$, the assumption on their mean squared distance is automatically satisfied whenever $\mcW_2(W_n(1),X_n(1))\to0$. This follows from the optimality of the comonotonic coupling. Next we state an upper bound on the quality of the auxiliary method based on comonotonic coupling of increments.

\begin{theorem}
\label{thm:X-B}
Let $X_n$ be a sequence of L\'evy processes satisfying~\eqref{eq:main_assumptions} with L\'evy measures $\Pi_n$ such that
\[
\mu_{4,n}\coloneqq\int_\mbR x^4\Pi_n(\D x)\to 0.
\]
Then for any sequence $k_n\to\infty$ we have
\[
\e\sup_{t\in [0,1]} \abs{X_n(t)-\wh W_n(t)}^2=\Oh(k_n\mu_{4,n}+\log k_n/k_n).
\]
\end{theorem}

\begin{proof}
An upper bound on the maximal distance $\sup_{t\in [0,1]} \abs{X_n(t)-\wh W_n(t)}$ is given by
\begin{equation}
\label{eq:triangle}
\max_{1\les i\les k_n} \abs{X_n(i/k_n)-\wh W_n(i/k_n)}+\sup_{t\in [0,1]} \abs{\wh W_n(t)-\wh W_n^{[k_n]}(t)}+\sup_{t\in [0,1]} \abs{X_n(t)-X_n^{[k_n]}(t)}.
\end{equation}
By construction, $(X(i/k_n)-\wh W(i/k_n),\; 1\les i\les k_n)$ is a zero-mean random walk, and so applying Doob's maximal inequality \cite[Prop.~7.16]{kallenberg} we obtain
\begin{equation}
\label{eq:max}
\e\max_{1\les i\les k_n} \abs{X(i/k_n)-\wh W(i/k_n)}^2\les 4\e\abs{X(1)-\wh W(1)}^2= 4k_n\e\abs{X(1/k_n)-\wh W(1/k_n)}^2.
\end{equation}
But the latter expectation is $\mcW_2^2(X(1/k_n),B(1/k_n))$ again by construction. According to Lemma~\ref{lem:Rio} we see that the upper bound in~\eqref{eq:max} is $\Oh(k_n\mu_{4,n})$. The other two terms in~\eqref{eq:triangle} are $\Oh(\log k_n/k_n)$ and $\Oh(k_n\mu_{4,n}+\log k_n/k_n)$ according to~\eqref{eq:BM_disc} and Proposition~\ref{prop:Levy_disc}, respectively.
\end{proof}

The following is our main result.

\begin{corollary}
\label{cor:Levy_to_BM}
Consider a sequence of L\'evy processes $X_n$ satisfying~\eqref{eq:main_assumptions} and $\mu_{4,n}\to 0$. Then for any $k_n\to\infty$ such that $k_n\mu_{4,n}\to 0$ we have
\begin{align*}
\e\sup_{t\in [0,1]} \abs{X_n(t)-W_n(t)}^2
&=\Oh\big(k_n\mu_{4,n}+\log k_n/k_n\big),
\end{align*}
provided $W_n(1)$ is chosen so that $\e(X_n(1)-W_n(1))^2$ is of the same order. 

In particular, taking $k_n\sim\sqrt{\abs{\log\mu_{4,n}}/\mu_{4,n}}$ and coupling $X_n(1)$ and $W_n(1)$ comonotonically yields
\[
\e\sup_{t\in [0,1]} \abs{X_n(t)-W_n(t)}^2
=\Oh\big(\log k_n/k_n\big)
=\Oh\big(\textstyle{\sqrt{\mu_{4,n}\abs{\log\mu_{4,n}}}}\big).
\]
\end{corollary}

\begin{proof}
Theorems~\ref{thm:basic} and~\ref{thm:X-B} yield the first part. The second part follows from Lemma~\ref{lem:Rio} implying $\e\abs{X_n(1)-W_n(1)}^2=\Oh(\mu_{4,n})$.
\end{proof}

\begin{remark}\label{rem:mu_bound}
Note that any upper bound $\ov\mu_{4,n}\to 0$ on $\mu_{4,n}$ can be taken instead in Corollary~\ref{cor:Levy_to_BM}. In that case it is sufficient to choose $k_n\sim c(\abs{\log\ov\mu_{4,n}}/\ov\mu_{4,n})^{1/2}$ for any $c>0$ to get the bound $\Oh((\ov\mu_{4,n}\abs{\log\ov\mu_{4,n}})^{1/2})$ on the mean squared maximal distance.
\end{remark}

It is generally not necessary to couple $X_n(1)$ and $W_n(1)$ comonotonically to obtain the rate in Corollary~\ref{cor:Levy_to_BM}.
This is important, since the distribution function of $X(1)$ is rarely explicit, and we may produce a near-comonotonic coupling by sampling independent copies of $X(1)$, see Lemma~\ref{lem:empirical_comonotonic} below.

Finally, we note that a slightly better rate can be produced when restricting comparison of the paths to the grid points:
\[
\e\max_{1\les i\les k_n} \abs{X_n(i/k_n)-W_n(i/k_n)}^2=\Oh\big(k_n\mu_{4,n}+\log\log k_n/k_n\big).
\]
This stems from a slightly better bound in Lemma~\ref{lem:Brownian_motions} as compared to Proposition~\ref{prop:bridges}.

\subsection{On the fourth moment}
Here we provide a general condition implying that $\mu_{4,n}\to 0$.

\begin{lemma}
\label{lem:mu4}
Let $X_n$ be a sequence of L\'evy processes converging to a Brownian motion and having L\'evy measures $\Pi_n$. Then $\mu_{4,n}\to 0$ if and only if
\begin{equation}
\label{eq:UI}
\lim_{M\to\infty} \limsup_{n\to\infty} \int_{\abs{x}>M} x^4\Pi_n(\D x)=0.
\end{equation}
\end{lemma}

\begin{proof}
Fix $h>0$.
According to~\cite[Thm.~15.14]{kallenberg}, the convergence $X_n(1)\cid W(1)$ implies
\[
\limsup_{n\to\infty} \int_{\abs{x}\les h} x^2\Pi_n(\D x)\les 1,
\qquad
\ov\Pi_n(h)\coloneqq\Pi_n(\mbR\setminus [-h,h])\to 0.
\]
Define the measures $\Pi'_n(\D x)\coloneqq\Pi_n(\D x)\1{\abs{x}> h}+ (1-\ov\Pi_n(h))\delta_0(\D x)$, which converge weakly to $\delta_0$, the point mass at~0, and are probability measures for all sufficiently large $n$. Thus
\[
\int_{\abs{x}>h} x^4\Pi_n(\D x)=\int_{\mbR} x^4\Pi'_n(\D x)\to 0,
\]
if and only if~\eqref{eq:UI} holds, which is the respective uniform integrability condition, see~\cite[Lem.~4.12]{kallenberg}.
It is left to note that $\limsup_{n\to\infty} \int_{\abs{x}\les h} x^4\Pi_n(\D x)\les h^2$, and to recall that $h>0$ was arbitrary.
\end{proof}

Some standard conditions implying~\eqref{eq:UI} can be also provided. For example, it is sufficient to assume that $\int_{\abs{x}>1} \abs{x}^{4+\delta}\Pi_n(\D x)$ is bounded for some small $\delta>0$. Alternatively, one can assume a bound on the tails: $\ov\Pi_n(x)\les\ov\Pi(x)$ for $x$ sufficiently large, where the L\'evy measure $\Pi$ satisfies $\int_{\abs{x}>1} x^4\Pi(\D x)<\infty$.

\section{Limiting regimes}
\label{sec:regimes}
Here we explore three limiting regimes: the classical scaling regime, perturbation of a Brownian motion by an independent  L\'evy process, and a small-jump Brownian approximation.
We write $W_n$ for the Brownian motion constructed from $X_n$ according to our increment reordering coupling, assuming that the end-points are coupled comonotonically.
In the latter two regimes it is more natural to index the sequence of processes by $\ve\downarrow 0$ instead of~$n\to\infty$.

\subsection{Zooming out}
For a L\'evy process $X$ satisfying condition~\eqref{eq:main_assumptions}, define a sequence of L\'evy processes $X_n(t)=X(nt)/\sqrt{n}$,\; $n\ges 1$. Then the associated L\'evy measures $\Pi_n$ satisfy
\[
\mu_{4,n}=\int_\mbR x^4\Pi_n(\D x)=\dfrac{1}{n}\int_\mbR x^4\Pi(\D x)\to 0,\quad n\to\infty.
\]
Thus, if $X_n(1)$ and $W_n(1)$ are comonotonically coupled and $k_n\sim c\sqrt{n\log n}$, for some $c>0$, then
\[
\e\sup_{t\in [0,1]} \abs{X_n(t)-W_n(t)}^2=\Oh\big(\sqrt{\log n/n}\big),\quad n\to\infty.
\]

The coupling in~\cite{khoshnevisan93} for a compensated Poisson process and a Brownian motion admits a similar rate.
More precisely, there is an a.s.\ upper bound of order $(\log\log n\log^2 n/n)^{1/2}$ on the squared maximal distance.
The same a.s. rate can be deduced for the coupling between a time-changed random walk $S_{[tn]}/\sqrt n$ with a Brownian motion as in~\cite{strassen1965}, given that $\e S_1^4<\infty$.

\subsection{Perturbed Brownian motion}
\label{subsec:perturbation}
Consider a standard Brownian motion $B$ perturbed by an independent L\'evy process $Y$ satisfying condition~\eqref{eq:main_assumptions}:
\[
X_\ve(t)=\sqrt{1-\ve^2}\;B(t)+\ve Y(t),\quad 0\les t\les 1,
\]
with $\ve\downarrow 0$. Letting $\Pi$ stand for the L\'evy measure of $Y$, we obtain
$
\mu_{4,\ve}=\ve^4\int_\mbR x^4\Pi(\D x)\to 0.
$
According to Corollary~\ref{cor:Levy_to_BM}, we choose $k_\ve\sim\sqrt{\abs{\log\ve}}/\ve^2$ to get
\[
\e\sup_{t\in [0,1]} \abs{X_\ve(t)-W_\ve(t)}^2=\Oh\big(\ve^2\sqrt{\abs{\log\ve}}\big),\quad \ve\to 0.
\]

Note that $\e\sup_{t\in [0,1]} \abs{X_\ve(t)-B(t)}^2=\Oh(\ve^2)$, which is smaller than the bound in the display by a logarithmic factor. To obtain this rate, we would need a simulatable $W_\ve$ that is sufficiently close to $B$. If the path of $X_\ve$ is already given, then this can be done by taking a sufficiently large $k_\ve$. Indeed, in this particular case, according to~\cite{recoverBM}, for fixed $\ve$ the Brownian bridge $W_\ve(t)-tW_\ve(1)$ converges in probability in supremum norm to the bridge $B(t)-tB(1)$ as $k_\ve\to\infty$. Letting $W'_\ve(t)=B(t)-t(B(1)-W_\ve(1))$ be the corresponding limiting process we find (after some straightforward manipulations) that
\[
\e\sup_{t\in [0,1]} \abs{X_\ve(t)-W'_\ve(t)}^2\les\Oh(\ve^2)+ 2\e\abs{X_\ve(1)-W_\ve(1)}^2=\Oh(\ve^2).
\]
Hence, the rate $\Oh(\ve^2)$ can be obtained by taking a sufficiently large $k_\ve$, increasing the cost. We stress that increasing $k_\ve$ does not always lead to an improvement. Indeed, if $X_\ve$ has no Brownian part, then infinite $k_\ve$ results in an independent bridge, as in the case of~$k_\ve=1$, see~\cite[Prop.~3]{recoverBM} for details.

\subsection{Small-jump Gaussian approximation}
\label{sec:ARA}
A widely used idea in simulation of L\'evy processes is to approximate the small jump martingale  by an appropriately scaled Brownian motion. For every cutoff level $\ve\in (0,1)$ we let $M_\ve(t)$ be the martingale containing the compensated jumps of $X$ in $[-\ve,\ve]$. We denote its variance at time $1$ by
\[
\sigma^2_\ve\coloneqq\int_{[-\ve,\ve]} x^2\Pi(\D x).
\]
According to~\cite{MR1834755} the process $X_\ve(t)=M_\ve(t)/\sigma_\ve$ weakly converges to $W(t)$ as $\ve\downarrow 0$ under a minor regularity condition, such as $\sigma_\ve\neq 0$ and $\ve/\sigma_\ve\to 0$, which we assume in the following. Below, we investigate the quality of our coupling in this limiting regime. 

Observe that
\[
\mu_{4,\ve}
=\sigma_\ve^{-4}\int_{[-\ve,\ve]} x^4\Pi(\D x)
\les\sigma_\ve^{-4} \ve^{2}
\int_{[-\ve,\ve]} x^2\Pi(\D x)
=\sigma_\ve^{-2}\ve^2.
\]
Corollary~\ref{cor:Levy_to_BM} (see also Remark~\ref{rem:mu_bound}), readily gives a bound on the mean squared maximal error:
\[
\e\sup_{t\in [0,1]} \abs{X_\ve(t)-W_\ve(t)}^2
=\Oh\big(\sigma^{-1}_\ve\ve\sqrt{\abs{\log(\sigma^{-1}_\ve\ve)}}\big),
\]
where $\abs{\log(\sigma^{-1}_\ve\ve)}=\Oh(\abs{\log\ve})$,  achievable by choosing $k_\ve\sim c\sigma_\ve\ve^{-1}\sqrt{\abs{\log\ve}}$ for some $c>0$. In words, the quality of our coupling is directly linked to the condition $\ve/\sigma_\ve\to 0$ and its rate of convergence.

Letting $\beta$ be the Blumenthal--Getoor index of $X$, i.e.
\begin{equation}
\label{eq:BG}
\beta=\inf\left\{p\ges 0:\, \int_{[-1,1]} \abs{x}^p\Pi(\D x)<\infty\right\}\in [0,2],
\end{equation}
we find that $\sigma^2_\ve\les\ve^{2-\beta_+}\int_{[-\ve,\ve]} x^{\beta_+}\Pi(\D x)=\Oh(\ve^{2-\beta_+})$ for any $\beta_+>\beta$. This readily yields an upper bound on the distance of scaled processes (as arising in applications):
\[
\e\sup_{t\in [0,1]} \abs{M_\ve(t)-\sigma_\ve W_\ve(t)}^2
=\Oh(\ve^{2-\beta_+/2}),
\qquad \text{with}\enskip
k_\ve\sim\ve^{-\beta/2}.
\]

More can be said under an additional lower bound assumption on the jump activity of $X$.
For instance, if $\ov\Pi(\ve)=\Pi(\mbR\setminus [-\ve,\ve])$ is regularly varying at~0, then the corresponding index of regular variation must be~$-\beta$, and by standard theory~\cite[Sec.~1.5 and~1.6]{BGT} we find that $\sigma^2_\ve$ is regularly varying with index $2-\beta$. Thus, $\ve/\sigma_\ve$ is regularly varying with index $\beta/2$ and we get an upper bound
\[
\e\sup_{t\in [0,1]} \abs{X_\ve(t)-W_\ve(t)}^2
=\Oh(\ve^{\beta_-/2}),\qquad \beta_-<\beta.
\]
In this case $k_\ve$, as prescribed by Remark~\ref{rem:mu_bound}, is regularly varying with index $-\beta/2$.


\section{Application to the multilevel Monte Carlo method}
\label{sec:multilevel}
Let us now return to the multilevel Monte Carlo method for the computation of $\e g(X)$ with $\e g^2(X)<\infty$ as discussed in the introduction.
Recall that $X_n$ is an approximation of $X$ obtained by replacing the martingale of jumps in $[-\ve_n,\ve_n]$ by a scaled Brownian motion with the same variance
\[
\sigma_n^2\coloneqq\int_{\abs{x}\les\ve_n} x^2\Pi(\D x).
\]
We choose a geometric sequence of truncation levels, say $\ve_n=2^{-n}$, and assume that the L\'evy measure $\Pi$ of $X$ has the Blumenthal--Getoor index $\beta\in [0,2]$, defined in~\eqref{eq:BG}. We will assume that $g$ is Lipschitz in supremum norm with a constant $L>0$.

\subsection{Coupling and level variance}
The crux of the method is to construct a pair $(X_n,X_{n+1})$ of successive approximations of $X$ so that the variance $\var[g(X_{n+1})-g(X_n)]$ is small.
Consider the decomposition
\[
X_{n+1}=M_n'+R_n',
\]
where $M_n'$ is the martingale of jumps in $[-\ve_n,-\ve_{n+1})\cup(\ve_{n+1},\ve_n]$ and $R_n'$ is the remainder, an independent process consisting of a Brownian motion and a drifted compound Poisson process with jumps exceeding $\ve_n$.
We form $X_n$ as an independent sum of $R'_n$ and $\sigma'_n W'_n$, where ${\sigma'_n}^2=\sigma^2_n-\sigma^2_{n+1}$ and $W'_n$ is a certain standard Brownian motion.
Note that the constructed $X_n$ indeed has the desired law.
Since $g$ is Lipschitz in the supremum norm, we have the upper bound
\begin{equation}
\label{eq:var}
\var[g(X_{n+1})-g(X_n)]
\les L^2\cdot\e\sup_{t\in [0,1]}\abs{M'_n(t)-\sigma'_nW'_n(t)}^2.
\end{equation}

\begin{remark}\label{rem:classical}
In the literature (see, e.g.~\cite{MR2759203,dereich_heidenreich_11}), the process $W'_n$ is typically an independent Brownian motion. Note that this still provides a coupling of $X_n$ and $X_{n+1}$ via the same $R'_n$. In this case, by Doob's martingale inequality, the bound in~\eqref{eq:var} is strictly of order $\Oh({\sigma'_n}^2)$, which is further upper bounded by $\Oh(\ve_n^{2-\beta_+})$ for any $\beta_+>\beta$.
\end{remark}

Here we propose to construct $W'_n$ from $X'_n=M'_n/\sigma'_n$ according to the algorithm with increment reordering presented in this paper with $k'_n$ as prescribed in Corollary~\ref{cor:Levy_to_BM}.
We assume that
$\sigma'_n/\ve_n\to\infty$, making the Brownian approximation of $X'_n$ valid in the limiting sense. 
Note that $X_n'$ is a compound Poisson process, which we may easily evaluate on any chosen grid. 
Furthermore, we assume that the comonotonic coupling of the end-points can be implemented with sufficient accuracy, and return to this issue later in \S\ref{subsec:near-comonotonic}. 
Finally, as in \S\ref{sec:ARA}, we find by Corollary~\ref{cor:Levy_to_BM} that 
\[
\e\sup_{t\in [0,1]} \abs{M'_n(t)-\sigma'_nW'_n(t)}^2
=\Oh\big(\sigma'_n\ve_n\sqrt{\abs{\log\ve_n}}\big)
=\Oh(\ve_n^{2-\beta_+/2}),
\qquad \beta_+>\beta,
\]
which implies the same upper bound on the level variance in~\eqref{eq:var}.
In particular, we have improved the variance by a factor $\ve_n/\sigma'_n\to 0$ up to a log term.
Furthermore, our chosen number of increments satisfies $k'_n\sim\sigma'_n\ve_n^{-1} \sqrt{\abs{\log\ve_n}}=\Oh(\ve_n^{-\beta_+/2})$.


\subsection{Computational complexity}
\label{subsec:complexity}
First, we consider the expected cost of sampling a pair of processes $(X_n,X_{n+1})$, where sampling does not include specification of $k'_n$ independent Brownian bridges for each process.
The expected number of jumps of $X_{n+1}$ is $\ov\Pi(\ve_{n+1})=\Oh(\ve_n^{-\beta_+})$, since we have assumed that $\ve_{n+1}=\ve_n/2$. 
This is also the expected cost of drawing all jumps and their times, and calculation of $k'_n$ increments of $M_n'$ as well as drawing the Brownian increments, since $k'_n=\Oh(\ve_n^{-\beta_+/2})$ is smaller. The reordering procedure incurs the cost $\Oh(k'_n\log k'_n)$, which is again of smaller order. Thus the cost associated to the level $n$ is of order $\Oh(\ve_n^{-\beta_+})$.
In other words, simulating a pair of coupled processes has nearly the same cost as simply drawing a sample of the marginal $X_{n+1}(1)$. 
In addition, we assume that $g(X_n)$ can be constructed from the above described skeleton at a comparable or lower cost. This is true for a number of functions $g$, including the functions  $g(X)=\int_0^1 X(t)\D t$ and $g(X)=\sup_{t\in [0,1]} X(t)$. 
In \S\ref{subsec:discr} below we consider evaluation of $g$ via further path discretization.

Secondly, to control the bias, we use~\cite[Cor.~6.2]{MR2759203}, which implies that
\[
\abs{\e g(X)-\e g(X_n)}=\Oh(\ve_n\abs{\log\ve_n}).
\]
Note that any coupling can be used to produce a weak bound here, as it does not need to be implementable. Attempting to control the bias with a square root of the second moment and using a bound from \S\ref{sec:ARA} leads to a slightly worse estimate $\Oh(\ve_n^{1-\beta_+/4})$.

Finally, we recall that $\ve_n=2^{-n}$ and apply~\cite[Thm.~2.1]{giles2015multilevel} (with $\alpha'<1$, $\beta'<2-\beta/2$ and $\gamma'>\beta$ arbitrarily close to their boundary values) to find an MLMC algorithm with mean squared error smaller than $\delta^2$ and computational complexity $\mcC_\delta$ satisfying
\[
\e\mcC_\delta
=\begin{cases}
\Oh(\delta^{-2}), &\beta<4/3,\\
\Oh(\delta^{-\frac{3}{2}\beta_+}), &\beta\ges 4/3.
\end{cases}
\]

\subsection{On discretization of paths}
\label{subsec:discr}
Even though the structure of $X_n$ allows to simulate $g(X_n)$ exactly for a number of functions~$g$, we assume that a discretization is used to approximate $g(X_n)$.
That is, we discretize the Brownian component of $X_n$ on a fine grid including the times $\{0,1/m_n,\ldots,1\}$ and the above sampled points of $X_n$. The error (both bias and level variance) of the resulting approximation $g(Y_n)$ may be controlled using an analogue of~\eqref{eq:BM_disc} below by virtue of the Lipschitz continuity of~$g$:
\[
\e\abs{g(Y_n)-g(X_n)}^2
\les L^2\cdot\e\sup_{t\in [0,1]} \abs{Y_n(t)-X_n(t)}^2
=\Oh(\log m_n/m_n).
\]
The orders of the bias, level variance and cost corresponding to the use of the approximations $Y_n$ to estimate $X$ are thus 
\[\Oh\big(\ve_n\abs{\log\ve_n}+\sqrt{\log m_n/m_n}\big),\qquad \Oh\big(\ve_n^{2-\beta_+/2}+\log m_n/m_n\big),\qquad\Oh\big(\ve_n^{-\beta_+}+m_n\big),\] respectively.
Taking $m_n$ much greater than $\ve_n^{-2}$ increases the cost without reducing the order of the bias or level variance. Similarly, taking $m_n$ much smaller than $\ve_n^{-\beta}$ is also suboptimal, since the bias grows without reducing the order of the level variance and cost.

Letting $m_n\sim\ve_n^{-p}$ for some $p\in [\beta,2]$ we may again apply~\cite[Thm.~2.1]{giles2015multilevel} (but with $\alpha'<p/2$, $\beta'<(2-\beta/2)\wedge p$ and $\gamma'>\beta\vee p$ arbitrarily close to their boundary values).
This leads to the minimization of $(\beta\vee p-(2-\beta/2)\wedge p)/p$ achieved by $p=\beta$, and the expected complexity bound
\[
\e\mcC_\delta
=\Oh(\delta^{-(5-4/\beta)\vee 2-\ve}),\quad \ve>0,
\]
for the required precision $\delta\downarrow 0$.
A more detailed analysis shows that the upper bound for $\beta<4/3$ can be improved to $\Oh(\delta^{-2}\abs{\log\delta}^3)$.
Since $m_n$ is of the same order as the expected number of jumps, it is, in fact, not strictly required to further discretize the path.
The randomness in jump times and their number will affect the small order term and not the power.

Finally, we stress that the standard way of taking an independent Brownian motion $W_n'$, as discussed in Remark~\ref{rem:classical}, results in the upper bounds $\e\mcC_\delta=\Oh(\delta^{-(6-4/\beta)\vee 2-\ve})$ and $\e\mcC_\delta=\Oh(\delta^{-(2\beta)\vee 2-\ve})$ for any $\ve>0$ when using $g(Y_n)$ and $g(X_n)$, respectively. Indeed, this follows from an analogous analysis and another application of~\cite[Thm.~2.1]{giles2015multilevel} (see also~\cite[Cor.~1.2]{MR2759203}). Thus our coupling leads to an improved computational complexity in the case when $\beta>1$, and this improvement is substantial for $\beta$ away from~1.
Graphical comparison of the respective powers is presented in Figure~\ref{fig:comparison}, where we also include the case when $g(X_n)$ can be sampled exactly.
For further comparison, it is noted that an ordinary MC results in $\e\mcC_\delta=\Oh(\delta^{-4-\ve})$ when paths are discretized and $\e\mcC_\delta=\Oh(\delta^{-2-\beta-\ve})$ when exact simulation of $g(X_n)$ is possible.

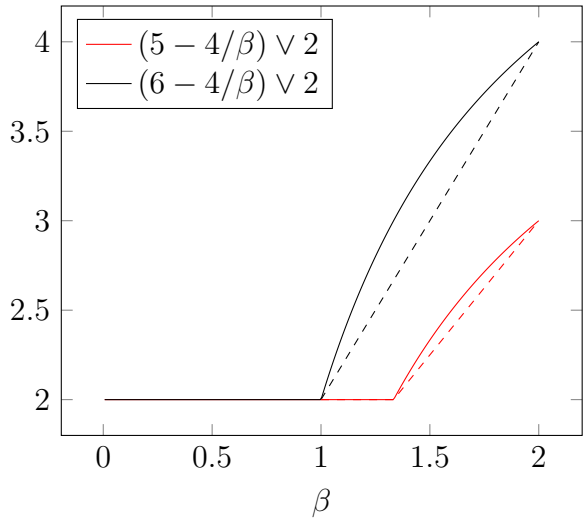
\begin{figure}[ht!]
\begin{tikzpicture}
\begin{axis}[xlabel = $\beta$,legend pos=north west,]
\addplot[domain=0:2, samples=300, color=red]{max(2,5-4/x)};
\addlegendentry{$(5-4/\beta)\vee 2$};
\addplot[domain=0:2, samples=300, color=black]{max(2,6-4/x)};
\addlegendentry{$(6-4/\beta)\vee 2$};
\addplot[domain=1:2, samples=300, color=black,style=dashed]{max(2,2*x)};
\addplot[domain=1:2, samples=300, color=red,style=dashed]{max(2,3*x/2)};
\end{axis}
\end{tikzpicture}
\caption{Asymptotic upper bounds on $\log \e\mcC_\delta/|\log \delta|$ in MLMC as precision $\delta\downarrow 0$. The coupling proposed in this paper is in red; solid (resp. dashed) lines correspond to the use of $g(Y_n)$ (resp. $g(X_n)$).}
\label{fig:comparison}
\end{figure}

\subsection{Near-comonotonic coupling}
\label{subsec:near-comonotonic}
Ideally we want to sample $X'_n(1)$ and $W'_n(1)$ comonotonically since this coupling minimizes the $L^2$-distance. Sampling from a coupling that does not increase the order of the $L^2$-distance is a rudimentary  fundamental problem. 
One way is described in Lemma~\ref{lem:empirical_comonotonic} and it consists of drawing a large number of independent copies of both random variables for each required sample.
This, however, comes at the expense of significantly increasing the computational cost. 
Simple analysis based on Corollary~\ref{cor:Levy_to_BM} and Lemma~\ref{lem:empirical_comonotonic} shows that the order of $\ve_n^{-\beta}$ of final values needs to be generated for each sample, unless one is willing to reuse these values while controlling the induced dependence.

Another way is to numerically evaluate the distribution function of $X'_n(1)$ including identification of the associated atoms if such exist.
In this regard, we observe that the Brownian component of $X_{n+1}$ could have been pushed into $M'_n$ instead of $R'_n$ when constructing the coupled pairs $(X_n,X_{n+1})$.
This change makes our problem more similar to the one studied in~\S\ref{subsec:perturbation} and affects neither the convergence rates nor the bounds on the MLMC computational complexity. 
It results in a smoothing effect, since now $X'_n$ has a Brownian component and the distribution of $X'_n(1)$ is continuous.
Fast Fourier inversion allows for computationally cheap evaluation of the distribution function with good numerical performance~\cite{MR2507761,MR1620156}.
It is, nevertheless, hard to incorporate the resulting numerical errors in our MLMC complexity analysis.


\section{Numerics}
\label{sec:numerical_simulations}
\subsection{Coupling implementation}
For the numerical experiments below we consider a L\'evy measure
\[
\Pi^0(\D x)=\left(0.4\abs{x}^{-\alpha-1}\1{x\in(-\ve_1,-\ve_2)}+ 0.6x^{-\alpha-1}\1{x\in(\ve_2,\ve_1)}\right)\D x,\qquad \alpha=1.5,
\]
with some truncation levels $0<\ve_2<\ve_1$.
Note that it corresponds to a truncation of the L\'evy measure of an $1.5$-stable process with skewness $0.2$.
The process $X$ is the respective drifted compound Poisson process with zero mean, rescaled so that $\e X^2(1)=1$.
The comonotonic coupling of the end-points $W(1)$ and $X(1)$ is implemented using a close approximation of the distribution function of the latter obtained from $30000$ independent realizations; it is fixed within every numerical experiment. A few smoothing options are possible here but they have no visible influence on the results below.
The root-mean-squared-maximal distance, the square root of~\eqref{eq:objective}, will be computed over the grid of mesh $2^{-12}\approx 0.00025$ using $1000$ independent samples of pairs of paths.

\begin{figure}[ht]
\centering
\includegraphics[width=0.45\textwidth]{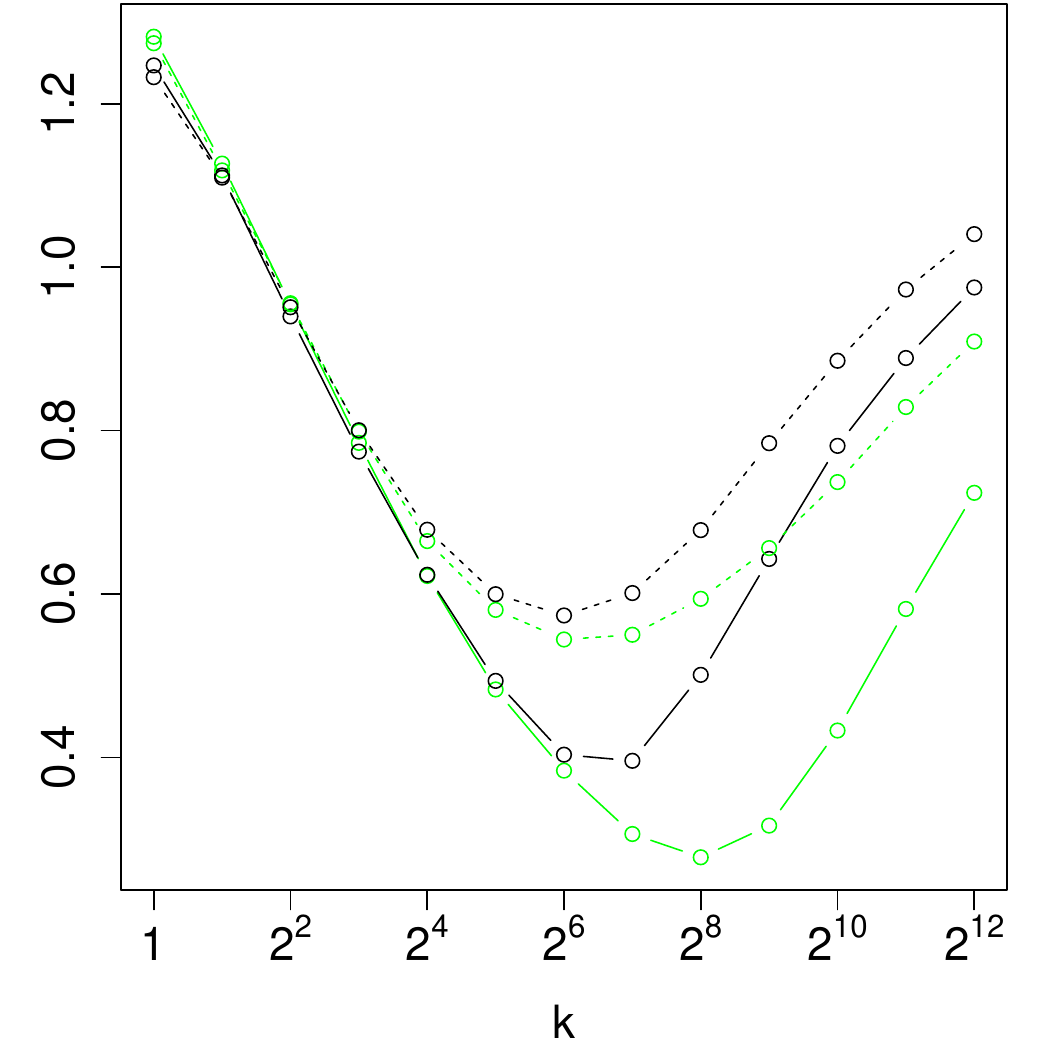}\qquad
\includegraphics[width=0.45\textwidth]{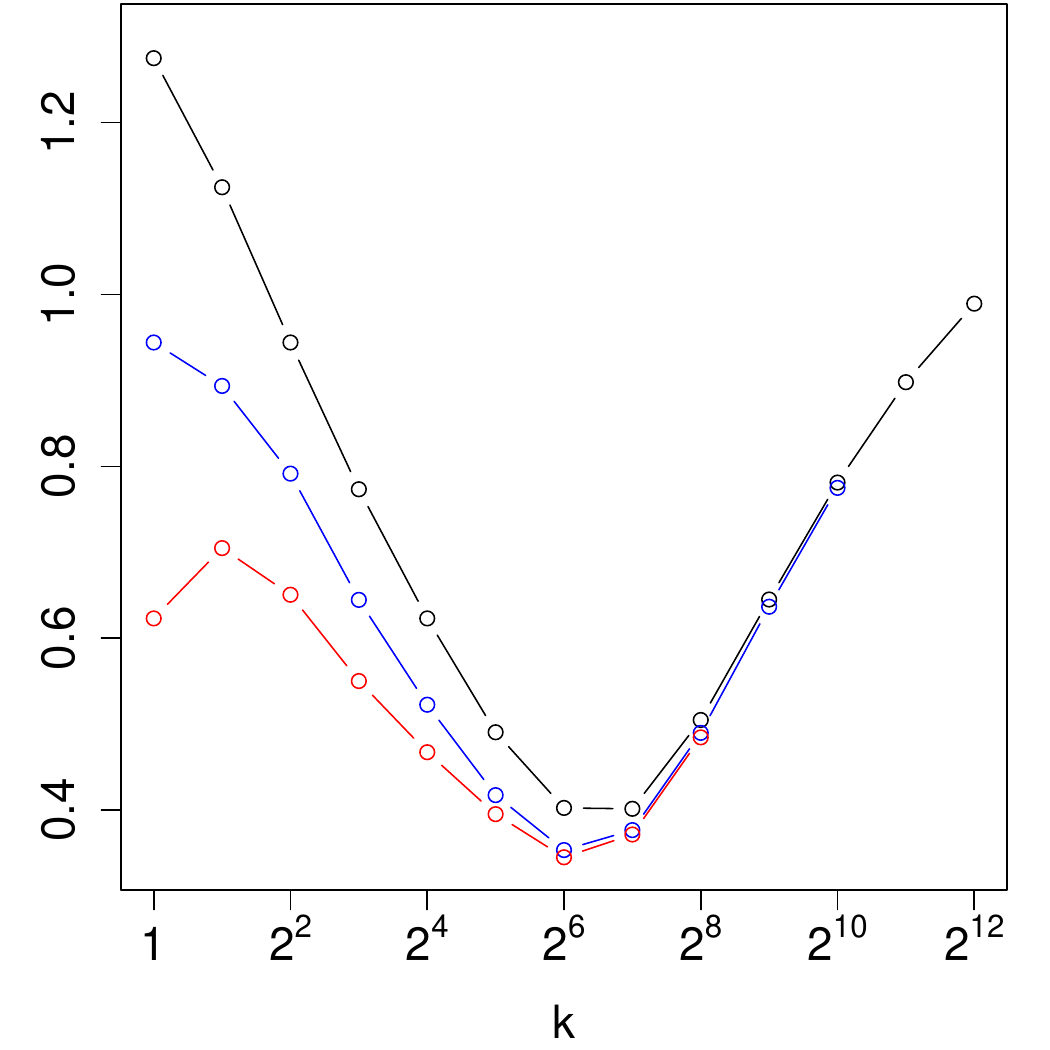}
\caption{Root-mean-squared-maximal distances between $X$ and $W$ as a function of~$k$.
Left: processes $X$ corresponding to $\ve_1\in\{0.1,1\}$ (solid, dashed) and $\ve_2\in\{0.01,0.03\}$ (green and black).
Right: second-level reordering with $\ve_1=0.1$, $\ve_2=0.03$ and $k_2\in\{1,4,16\}$ (black, blue, red).}
\label{fig:exp12}
\end{figure}

In our first experiment we investigate the quality of our coupling for various values of the parameter~$k$, which is the number of incremental processes to be permuted, see~\S\ref{sec:construction}.
We take $\ve_1\in\{0.1,1\}$ (solid, dashed) and $\ve_2\in\{0.01,0.03\}$ (green and black), resulting in four different processes~$X$.
We apply our coupling for $k=2^i$, $i=0,\ldots,12$, and plot in Figure~\ref{fig:exp12} (left) the estimated root-mean-squared-maximal distance.
Observe that the optimal $k$ (among powers of 2) is larger for processes better approximating the Brownian motion, that is, when $\ve_1,\ve_2$ are smaller.
As explained in the Introduction, the case of a very large $k$ need not be good, and should eventually result in the same error as~$k=1$.

In our second experiment we take $\ve_1=0.1,\ve_2=0.03$ (solid black) and apply the reordering idea on two levels, see Remark~\ref{rem:hierarchical}.
Firstly, we use $k=2^i$ as above, and then reorder $k_2\in\{1,4,16\}$ (black, blue, red) increments in each of the $k$ pieces. Note that $k_2=1$ corresponds to the first experiment.
The corresponding root-mean-squared-maximal distances are presented in Figure~\ref{fig:exp12} (right). In this case the second-level reordering is beneficial, and the optimal $k$ is similar in the three considered scenarios.
The smallest value is $0.40$ for the standard coupling with $k=2^7$ ($k=2^6$ gives almost the same result) and $0.34$ for the two-level procedure with $k=2^6,k_2=16$, which is about $15\%$ less.
In applications, one may first use our standard coupling and find a good $k$ and then try second-level reordering for a few~$k_2$.

Next, we illustrate the coupling for the process $X$ (still corresponding to $\ve_1=0.1,\ve_2=0.03$) using the above found optimal $k=2^7$.
Figure~\ref{fig:path} (top left) presents a histogram of maximal absolute distances resulted from $1000$ independent trials, which indeed has root-mean-squared value of $0.40$.
Furthermore, we plot the pairs of paths corresponding to the $0.05,0.50,0.95$ quantiles (top-right to bottom-right).
Importantly, the root-mean-squared error at the end-point, $\sqrt{\e (W(1)-X(1))^2}$, is estimated to be $0.008$, which is negligible when compared to the error for the sample paths.

\begin{figure}[ht!]
\centering
\includegraphics[width=0.45\textwidth]{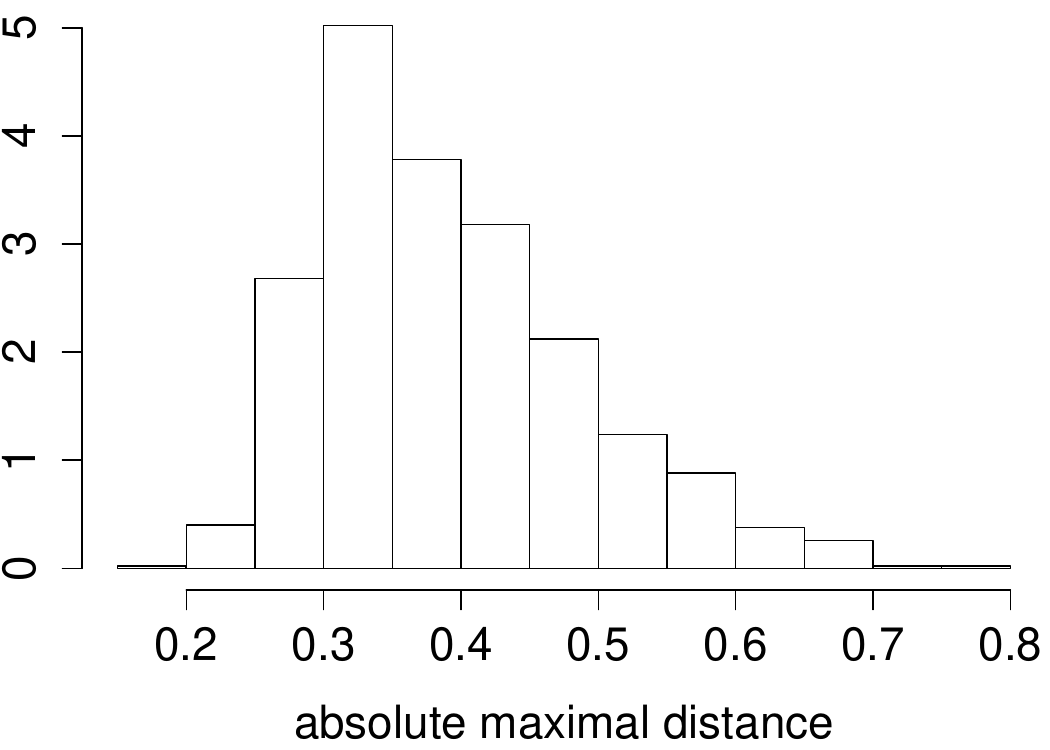}\quad
\includegraphics[width=0.45\textwidth]{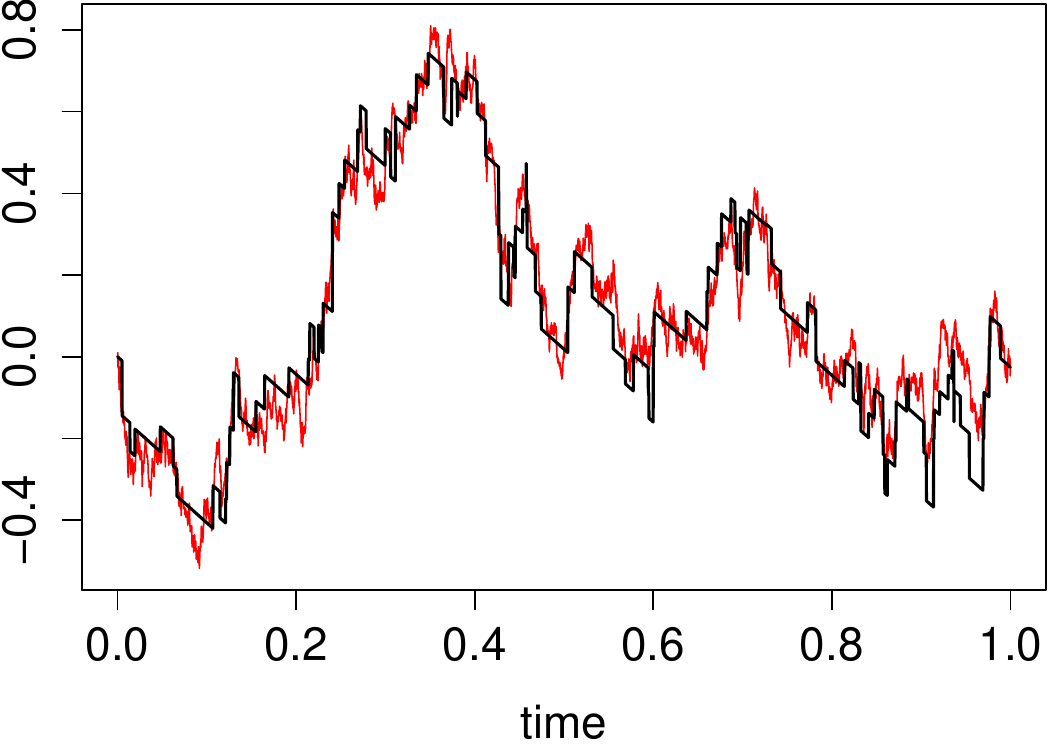}\\
\medskip
\includegraphics[width=0.45\textwidth]{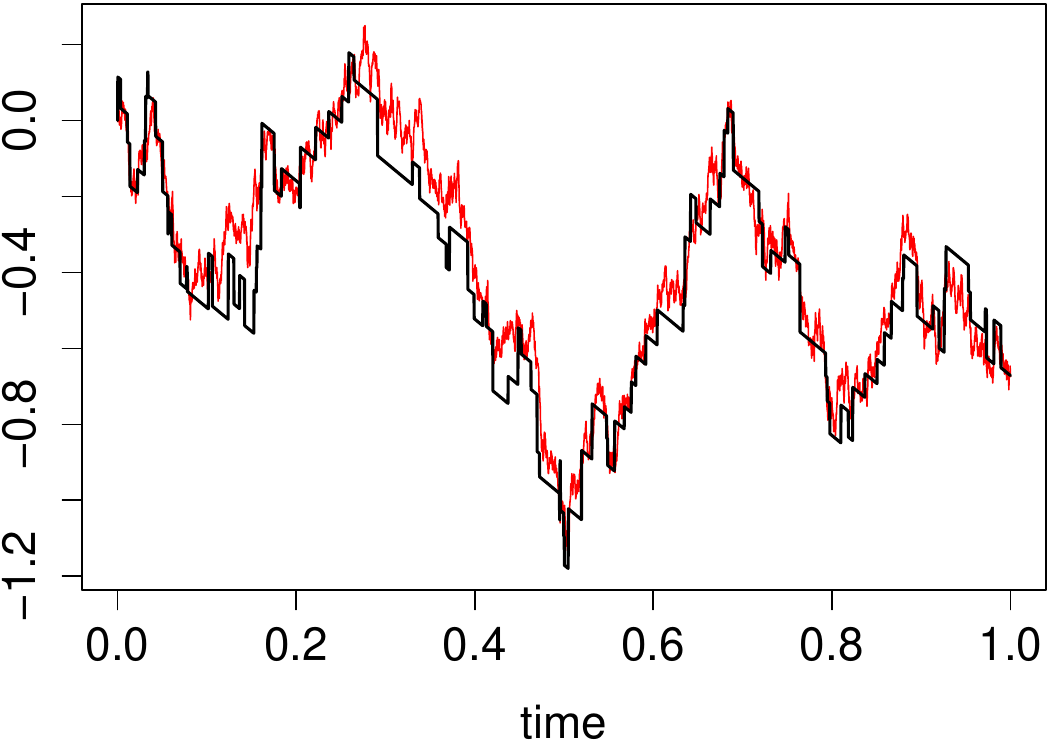}\quad
\includegraphics[width=0.45\textwidth]{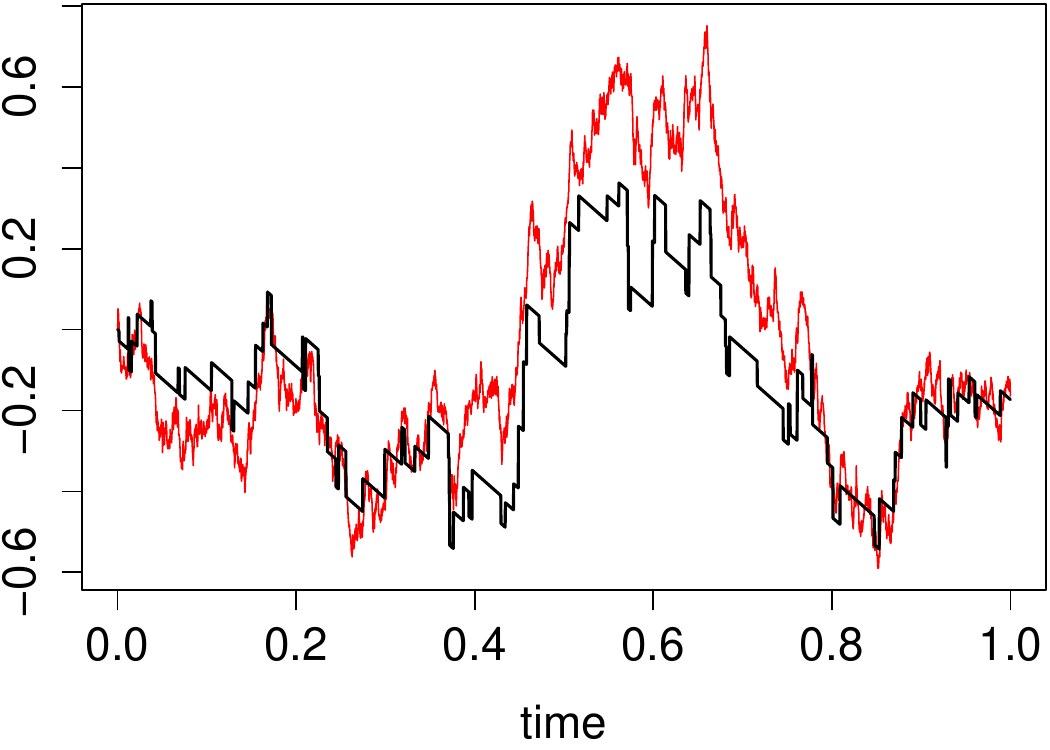}
\caption{Illustration of coupling using $k=2^7$: histogram of maximal absolute distances and three pairs of paths corresponding to $0.05,0.50,0.95$ quantiles.}
\label{fig:path}
\end{figure}

Finally, we provide an illustration of the limit result in Corollary~\ref{cor:Levy_to_BM} in the setting similar to \S\ref{sec:multilevel}.
Take a geometric sequence of truncation levels $\ve_{1,n}=2^{-n}$ and $\ve_{2,n}=2^{-n-1}$, and let $X_n$ correspond to a rescaled martinagle of jumps in $(-2^{-n},-2^{-n-1})\cup(2^{-n-1},2^{-n})$ as above.
We estimate the root-mean-squared-maximal distance for our coupling using different choices of $k_n$ (powers of $2$) and denote the smallest such distance by $d^*_n$ and the respective $k_n$ by $k^*_n$.
These are plotted in Figure~\ref{fig:lim} and also compared to their theoretical counterparts in Corollary~\ref{cor:Levy_to_BM}.
We find a rather good prediction of both the error and the adequate number of increments, which suggests that our theoretical upper bound is rather tight.

\begin{figure}[ht]
\centering
\includegraphics[width=0.45\textwidth]{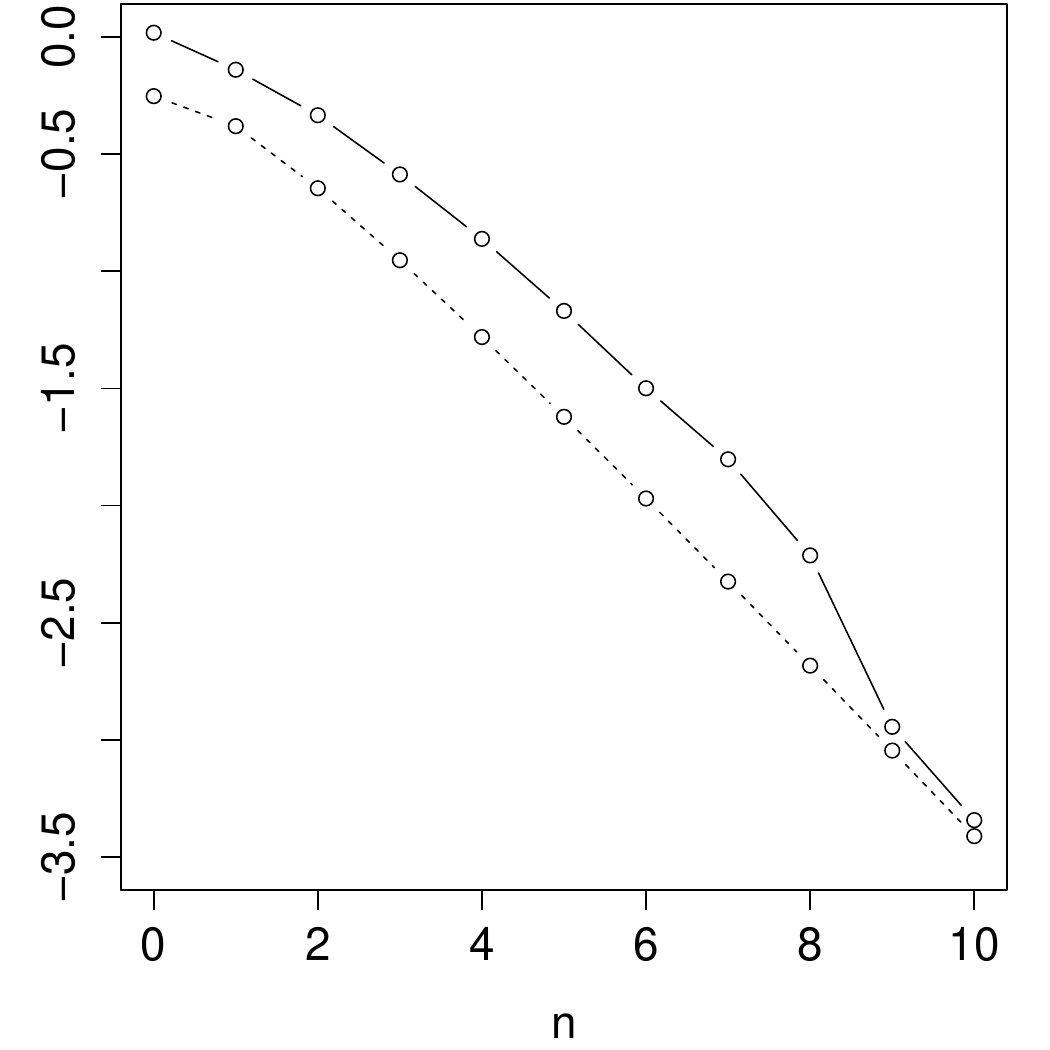}\qquad
\includegraphics[width=0.45\textwidth]{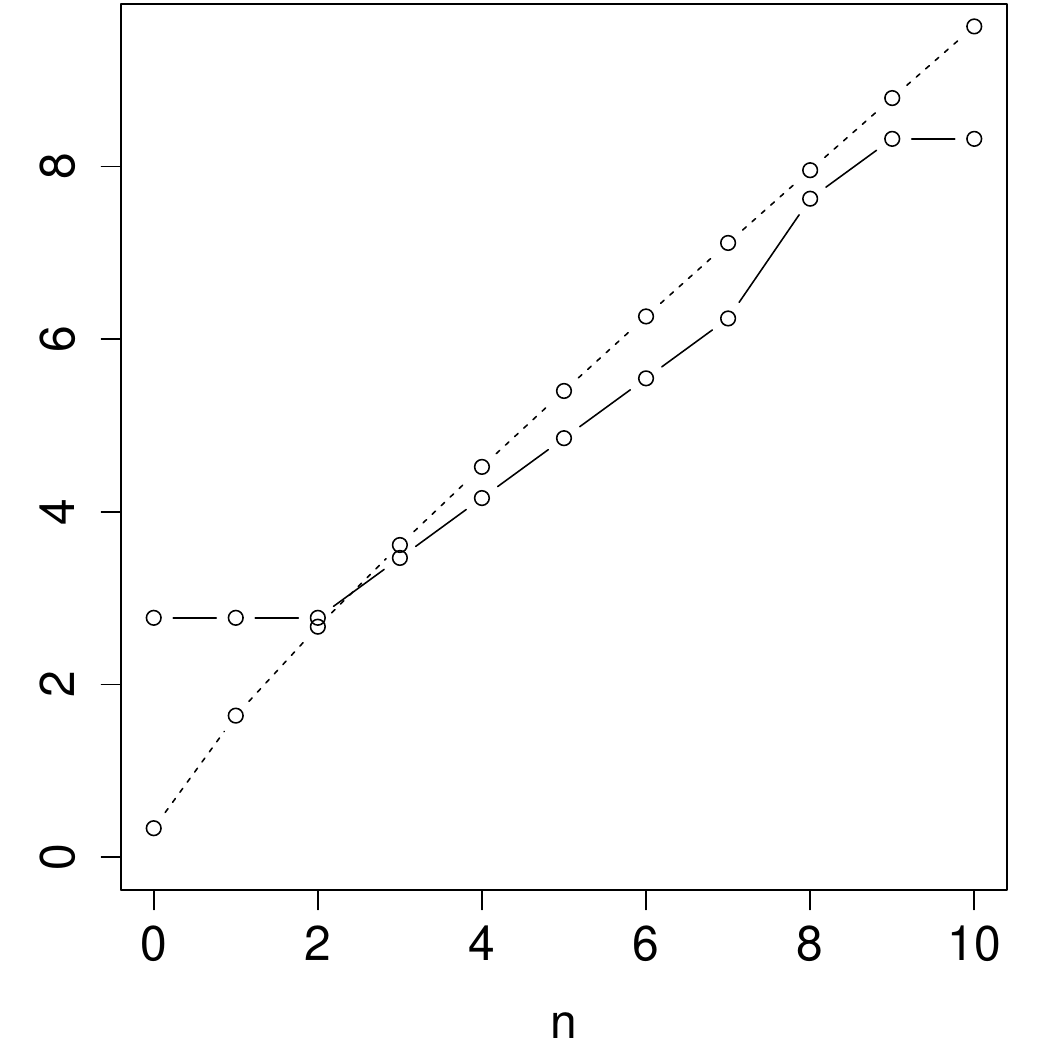}
\caption{The optimal root-mean-squared-maximal distance $d^*_n$, the respective number $k^*_n$, and their theoretical counterparts for $\ve_{1,n}=2^{-n}$ and $\ve_{2,n}=2^{-n-1}$.
Left: $\log d^*_n$ (solid) and $\log(\mu_{4,n}|\log \mu_{4,n}|)/4$ (dashed). Right: $\log k^*_n$ (solid) and $\log(|\log \mu_{4,n}|/\mu_{4,n})/2$ (dashed).}
\label{fig:lim}
\end{figure}

In conclusion, we would like to stress that implementation of the coupling presented in this paper is rather straightforward, but, nevertheless, some issues may arise.
The increments of a compound Poisson process will normally exhibit ties which must be resolved randomly.
Furthermore, for a drifted process such ties may not be detected (numerical rounding) leading to strange non-Brownian trajectories.
A simple solution of this numerical problem is to add a negligible Brownian component to the process.

\subsection{Multilevel Monte Carlo}
To test the performance of our coupling algorithm within the context of the multilevel Monte Carlo analysis of \S\ref{sec:multilevel}, we will consider a tempered stable L\'evy process with zero mean and L\'evy measure
\[
\Pi(\D x) = 0.05|x|^{-\alpha-1}e^{-|x|}\D x,\qquad \alpha\in\{1.2,1.5\},
\]
and the truncation levels $\ve_n=2^{-n/4}$. To couple the endpoints $X'_n(1)$ and $W_n'(1)$, we will take the samples we produced for the estimation and match the pairs by rank order, as discussed in \S\ref{subsec:near-comonotonic}. To test the empirical performance of the multilevel Monte Carlo estimation, it suffices to verify that the bias and level variance converge at the predicted rate or faster. Testing the accuracy of our estimation would require access to the value of $\e g(X)$, which is rarely available analytically. 
To obtain a numerical approximation, we consider the function $g(X)=\sup_{t\in[0,1]}X(t)$ and use the methodology proposed in~\cite{TSBA}, which is more limited but very efficient for this example. 

Following the analysis of \S\ref{sec:multilevel}, we chose $k'_n\sim\ve_n^{-\alpha/2}\sqrt{|\log\ve_n|}$ and $m_n\sim\ve_n^{-p}$. In multilevel Monte Carlo, the number of samples per level depends on the desired accuracy of the estimator, the level variance and simulation cost, they decrease exponentially in $n$ and are typically given adaptively as the simulations are run. For simplicity, we drew $10^{6-\lfloor n/15\rfloor}$ samples of level $n$ for each $n$. Figure~\ref{fig:mlmc} shows the convergence rate of the bias and level variance of the Monte Carlo estimate in \S\ref{sec:multilevel} as a function of the truncation levels $\varepsilon_n$ with the top (resp. bottom) figures corresponding to the case $\alpha=1.2$ (resp. $\alpha=1.5$). The dashed lines indicate the predicted rate of convergence in both cases (see details in \S\ref{sec:multilevel}). The bias and level variance both satisfy the stated bounds, with a particularly good agreement between the the level variance and its bound. 

\begin{figure}[ht]
\centering
\includegraphics[width=0.45\textwidth]{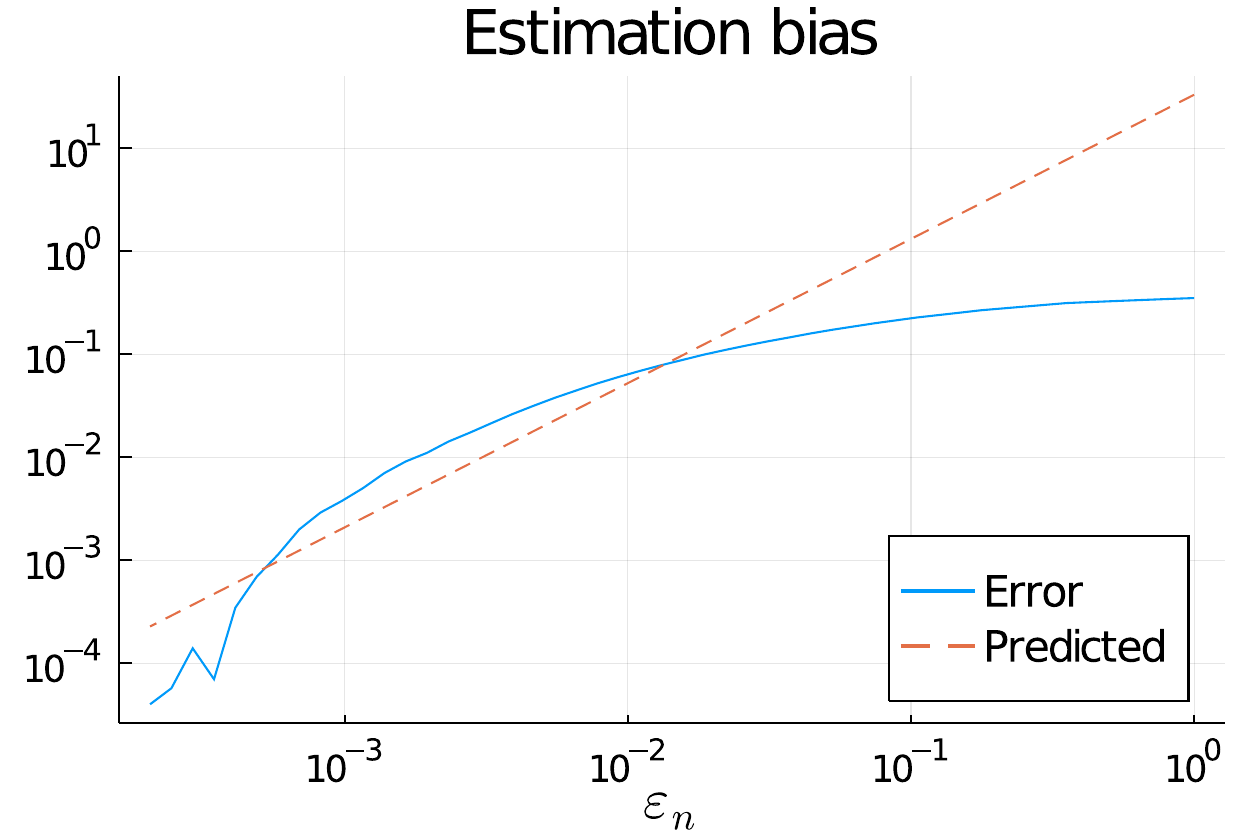}\qquad
\includegraphics[width=0.45\textwidth]{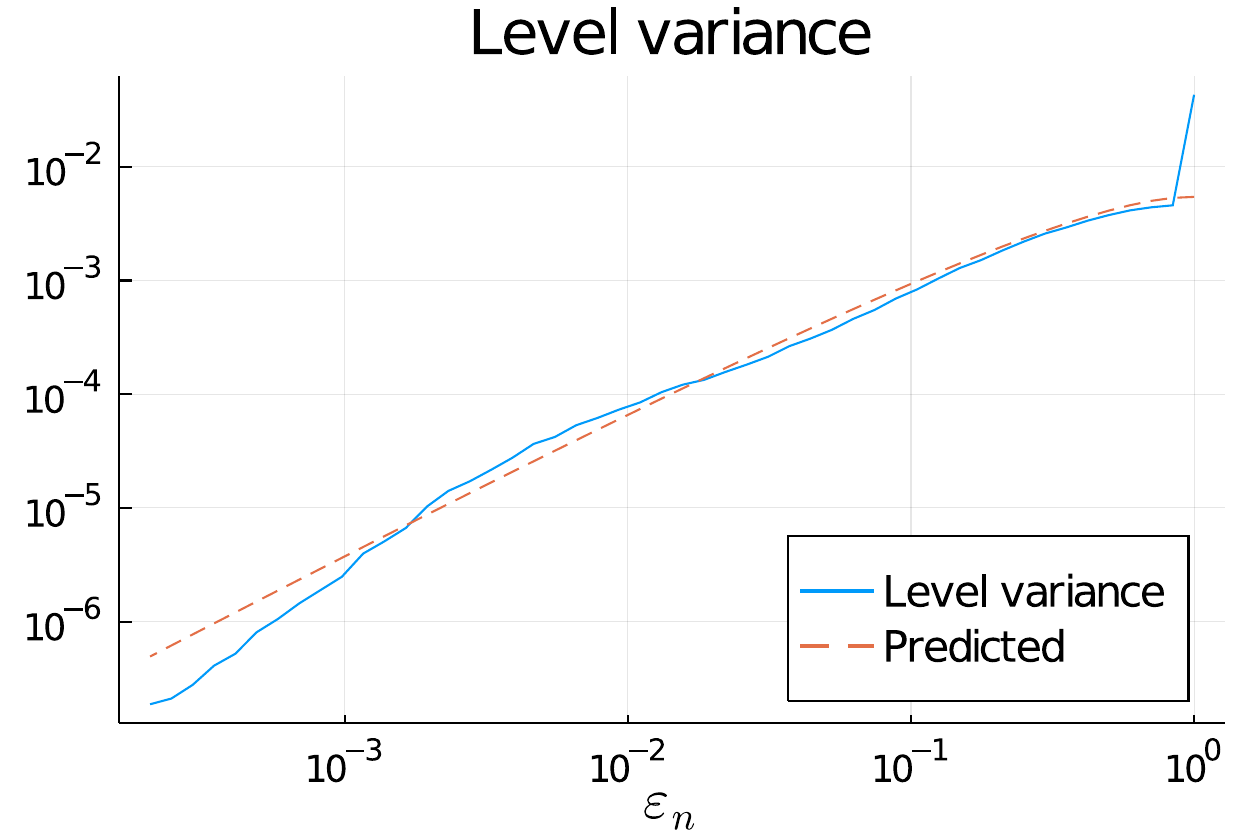}
\includegraphics[width=0.45\textwidth]{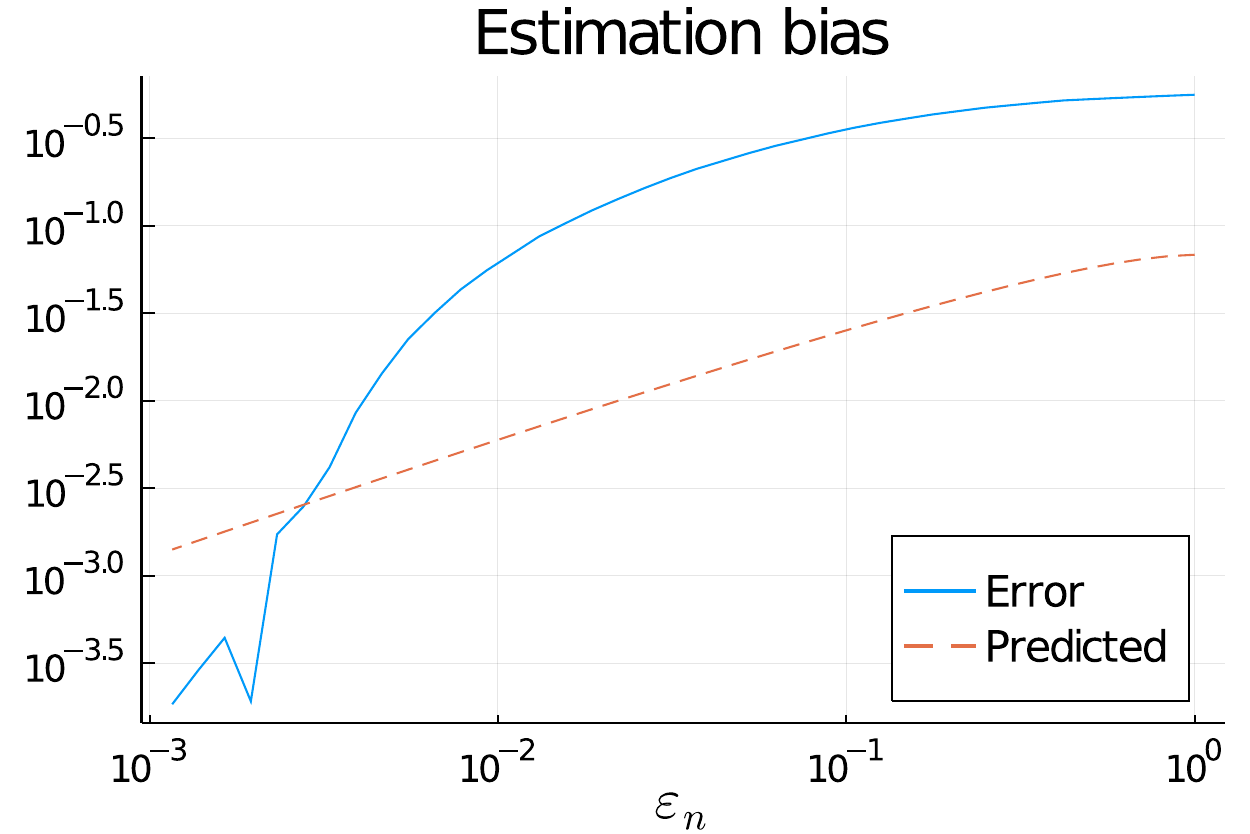}\qquad
\includegraphics[width=0.45\textwidth]{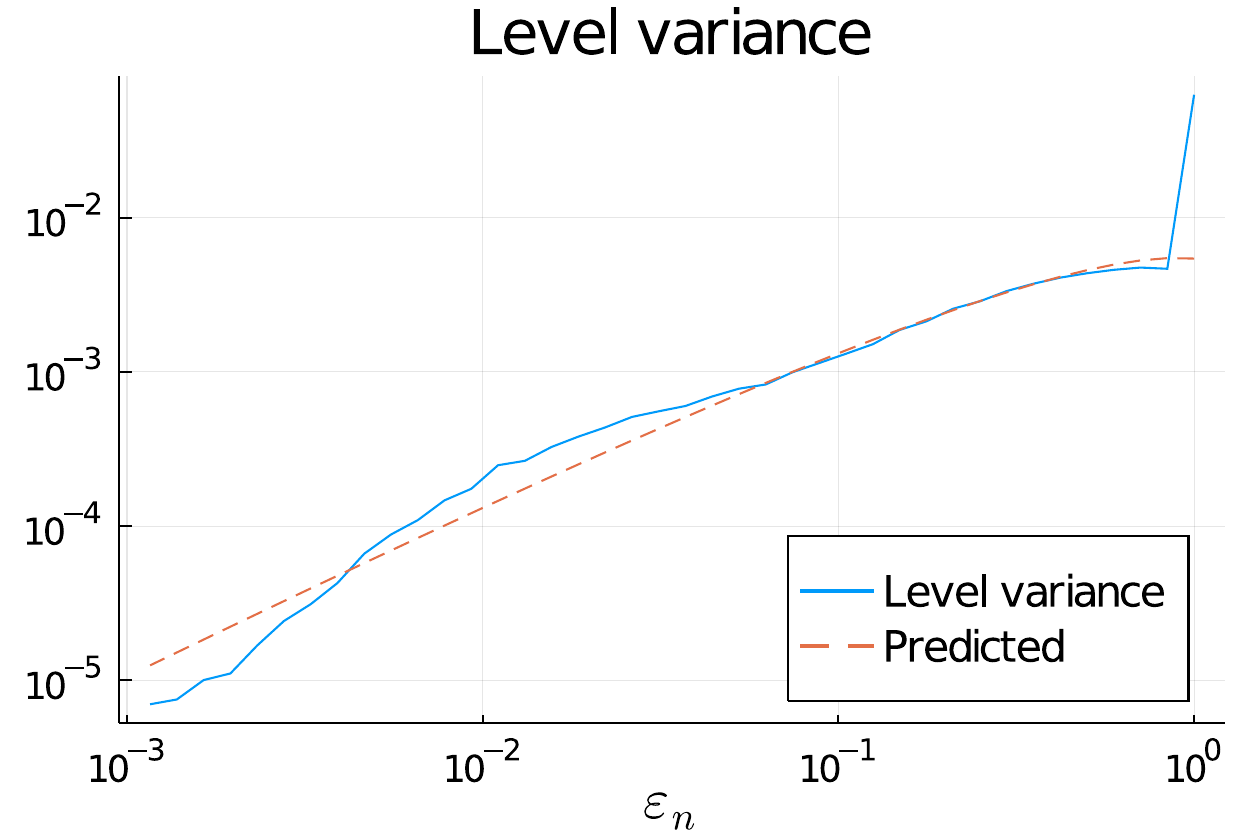}
\caption{The figures show the bias and level variance of the multilevel Monte Carlo method discussed in \S\ref{sec:multilevel}. Top: $\alpha=1.2$, bottom: $\alpha=1.5$, left: decay of the bias $\e (g(X)- g(Y_n))$ (solid) and its bound $\ve_n^{\alpha/2}\sqrt{|\log\ve_n|}$ (dashed), right: decay of the level variance $\var(g(Y_{n+1})-g(Y_n))$ (solid) and its bound $\max\{\ve_n^{2-\alpha/2},\ve_n^{\alpha}|\log\ve_n|\}$ (dashed).}
\label{fig:mlmc}
\end{figure}

\appendix
\section{Bounds for discretized processes}
\label{app:bound}
Proposition~\ref{prop:Levy_disc} provides an upper bound for a similar discretization error for a L\'evy process. Its proof is based on the following three auxiliary results.

\begin{lemma}
\label{lem:tail_estimate}
Let $\ve>0$ and $X$ be a L\'evy process without a Brownian component with $\e X(1)=0$, $\e X^2(1)\les 1$ and the corresponding L\'evy measure $\Pi$ supported on $[-\ve,\ve]$. Then for any $t,x,u>0$ we have
\[
\p\Big(\sup_{s\in [0,t]} X(s)\ges x\Big)\les\exp\big(-ux+u^2e^{u\ve}t\big).
\]
\end{lemma}

\begin{proof}
The inequality $e^z-1-z\les z^2e^{\abs{z}}$ yields
\[
\Psi(u)\coloneqq\int_{[-\ve,\ve]} (e^{uz}-1-uz)\Pi(\D z)
\les u^2e^{u\ve}\int_{[-\ve,\ve]} z^2\Pi(\D z)
\les u^2e^{u\ve}.
\]
Since $\exp(uX(t)-t\Psi(u))$ is a martingale and $\Psi(u)$ is non-negative, $\exp(u X(t))$ is a submartingale and Doob's maximal inequality~\cite[Prop.~7.15]{kallenberg} yields
\[
\p\bigg(\sup_{s\in [0,t]} X(s)\ges x\bigg)
=\p\big(e^{u\sup_{s\in [0,t]} X(s)}\ges e^{ux}\big)
\les e^{-ux}\e e^{uX(t)}=e^{-ux+t\Psi(u)}
\]
implying the stated bound.
\end{proof}

\begin{lemma}
\label{lem:mart_disc}
There is a constant $C>0$ such that
\[
\e\sup_{t\in [0,1]} \big(X(t)-X^{[k]}(t)\big)^2\les C\log k/k
\]
for all $k\ges 2$ and every L\'evy process $X$ satisfying the conditions of Lemma~\ref{lem:tail_estimate} with $\ve\les (k\log k)^{-1/2}$.
\end{lemma}

\begin{proof}
Let $E_1,E_2,\ldots$ be i.i.d.\ standard exponential random variables. For any $u,x>0$ we have
\[
\p(E_1/u+ue^{u\ve}/k\ges x)=\exp(-ux+u^2e^{u\ve}/k)\wedge 1,
\]
which is also an upper bound on $\p(\sup_{s\in [0,1/k]} X(s)\ges x)$ by Lemma~\ref{lem:tail_estimate}. Thus, $\sup_{s\in [0,1/k]} X(s)$ and similarly $-\inf_{s\in [0,1/k]} X(s)$ are stochastically bounded by $E_1/u+ue^{u\ve}/k$.

Let $(\ov G_i,\und G_i)$ be independent copies of $(\sup_{s\in [0,1/k]} X(s),-\inf_{s\in [0,1/k]} X(s))$. Then
\[
\e\sup_{t\in [0,1]} \big(X(t)-X^{[k]}(t)\big)^2
=\e\max_{1\les i\les k}(\ov G_i\vee\und G_i)^2
\les\e\max_{1\les i\les k}(\ov G^2_i+\und G_i^2)
\les\e\max_{1\les i\les k}\ov G^2_i
+\e\max_{1\les i\les k}\und G^2_i.
\]
But each of the latter is upper bounded by
\[
\e\max_{1\les i\les k} (E_i/u+ue^{u\ve}/k)^2\les\dfrac{2}{u^2}\e\max_{1\les i\les k} E_i^2+\dfrac{2u^2}{k^2}e^{2u\ve}.
\]
It is a basic fact that
\[
\max_{1\les i\les k} E_i-\log k\cid G,\quad k\to\infty,
\]
where $G$ has the standard Gumbel distribution. From~\cite{pickands_68} we also have
\[
\e\Big(\max_{1\les i\les k} E_i-\log k\Big)\to\e G,
\qquad
\e\Big(\max_{1\les i\les k} E_i-\log k\Big)^2\to\e G^2,
\]
implying $\e\max_{1\les i\les k} E_i^2=\log^2 k+\Oh(\log k)$. Hence we may choose a constant $C'$ such that
\[
\e\sup_{t\in [0,1]} \big(X(t)-X^{[k]}(t)\big)^2
\les C'\frac{2}{u^2}\log^2 k+\frac{2u^2}{k^2}e^{2u\ve}
\]
for all $k\ges 2$ and all $X$ satisfying our assumptions. Choose $u=\sqrt{k\log k}$ to get the result.
\end{proof}

\begin{lemma}
\label{lem:CPP_disc}
There is a universal constant $C>0$ such that
\[
\e\sup_{t\in [0,1]} \big(X(t)-X^{[k]}(t)\big)^2
\les C(\ve^{-2}\mu_4/\sqrt{k}+\sqrt{\mu_4}),
\]
for any $k\ges 1$, any $\ve>0$ and any compensated compound Poisson process $X$ with jumps of size at least $\ve$.
\end{lemma}

\begin{proof}
Applying Doob's martingale inequality~\cite[Prop.~7.16]{kallenberg}, we obtain
\[
\e\sup_{t\in [0,1]} \big(X(t)-X^{[k]}(t)\big)^4\les k\e\sup_{t\in[0,1/k)} X^4(t)\les (4/3)^4\cdot k\e X^4(1/k).
\]
Using the moment formula in terms of cumulants and the equality $\e X(1)=0$, we obtain
\[
\e X^4(t)=3t^2\left(\int_{\abs{x}\ges\ve} x^2\Pi(\D x)\right)^2+t\int_{\abs{x}\ges\ve} x^4\Pi(\D x),
\]
and noting that
\[
\int_{\abs{x}\ges\ve} x^2\Pi(\D x)\les\ve^{-2}\int_{\abs{x}\ges\ve} x^4\Pi(\D x)=\ve^{-2}\mu_4,
\]
we get
\[
k\e X^4(1/k)\les 3\ve^{-4}\mu_4^2/k+\mu_4.
\]
By Jensen's inequality,
\[
\e\sup_{t\in [0,1]} \big(X(t)-X^{[k]}(t)\big)^2\les 16/9\cdot \big(3\ve^{-4}\mu_4^2/k+\mu_4\big)^{1/2},
\]
and the result follows.
\end{proof}

\begin{proof}[Proof of Proposition~\ref{prop:Levy_disc}]
Let $\ve=(k\log k)^{-1/2}$ and consider the L\'evy--It\^o decomposition $X(t)=\sigma B(t)+Y(t)+Z(t)$, where $B$ is a standard Brownian motion, $Y$ is a compensated compound Poisson process with jumps of size exceeding $\ve$, and $Z$ is a pure-jump martingale with jumps of size at most~$\ve$. Since $\sigma^2\les\e X^2(1)\les 1$, we have
\[
\sup_{t\in [0,1]} \abs{X(t)-X^{[k]}(t)}\les\sup_{t\in [0,1]} \abs{B(t)-B^{[k]}(t)}+\sup_{t\in [0,1]} \abs{Y(t)-Y^{[k]}(t)}+\sup_{t\in [0,1]} \abs{Z(t)-Z^{[k]}(t)}.
\]
Bounding each term in the squared mean sense by using relation~\eqref{eq:BM_disc}, Lemma~\ref{lem:mart_disc} and Lemma~\ref{lem:CPP_disc}, for all $k\ges 2$ we get the bound
\[
C(\sqrt{k}\log k\cdot\mu_4+\sqrt{\mu_4}+\log k/k).
\]
It is left to note that $\log k<\sqrt{k}$ and $\mu_4/(k\mu_4+\log k/k)^2\les 1/(2\log k)\les 1/(2\log 2)$.
\end{proof}

We note that Proposition~\ref{prop:Levy_disc} holds under less restrictive assumptions.
It is sufficient to assume instead of~\eqref{eq:main_assumptions} that $\e X^2(1)\les 1 ,\e X^4(1)<\infty$, since the additional drift introduces an error of order $[\e X(1)/k]^2$, which does not change the bound.

\section{Near-comonotonic coupling via repeated simulations}
\label{sec:final_value}
\begin{lemma}
\label{lem:empirical_comonotonic}
Let $\xi,\xi_1,\xi_2,\ldots$ and $\zeta,\zeta_1,\zeta_2,\ldots$ be 
two independent i.i.d. sequences with laws $\mu_\xi$ and 
$\mu_\zeta$, respectively, with finite fourth moment. 
For any $n\ges 2$ let $\xi_{(1)}\les\ldots\les\xi_{(n)}$ and 
$\zeta_{(1)}\les\ldots\les\zeta_{(n)}$ be the ranked values of 
$\xi_1,\ldots,\xi_n$ and $\zeta_1,\ldots,\zeta_n$, respectively. 
Then for any independent uniform random variable $U$ on 
$\{1,\ldots,n\}$, we have $\xi_{(U)}\sim\mu_\xi$, 
$\zeta_{(U)}\sim\mu_\zeta$ and
\[
\e\big(\xi_{(U)}-\zeta_{(U)}\big)^2
\les 3\mcW_2^2(\mu_\xi,\mu_\zeta)
+\frac{c\log n}{\sqrt{n}}
\big(\sqrt{\e\xi^4}+\sqrt{\e\zeta^4}\big),
\]
where $c>0$ is a universal constant independent of $\mu_\xi$,
$\mu_\zeta$ and $n$.
\end{lemma}

\begin{proof}
The independence and exchangeability imply that  $\xi_{(U)}\sim\mu_\xi$ and $\zeta_{(U)}\sim\mu_\zeta$. Define the empirical distribution functions $\mu^\n_\xi\coloneqq n^{-1}\sum_{k=1}^n \delta_{\xi_k}$ and $\mu^\n_\zeta\coloneqq n^{-1}\sum_{k=1}^n \delta_{\zeta_k}$. By~\cite[Cor.~7.18]{MR4028181} (and the discussion preceding it), there exists a universal constant $c>0$ such that 
then
\begin{equation}
\label{eq:W-empirical}
\e\mcW_2^2(\mu^\n_\xi,\mu_\xi)
\les\dfrac{c\log n}{\sqrt{n}}\sqrt{\e\xi^4},
\quad\text{and}\quad
\e\mcW_2^2(\mu^\n_\zeta,\mu_\zeta)
\les\dfrac{c\log n}{\sqrt{n}}\sqrt{\e\zeta^4},
\qquad
\forall\, n\ges 2.
\end{equation}

Note that
\begin{align*}
\e\big(\xi_{(U)}-\zeta_{(U)}\big)^2
&=\dfrac{1}{n}\sum_{k=1}^n
\e(\xi_{(k)}-\zeta_{(k)})^2=\e\mcW_2^2(\mu^\n_\xi,\mu^\n_\zeta)\\
&\les 3\mcW_2^2(\mu_\xi,\mu_\zeta)
+3\e\mcW_2^2(\mu^\n_\xi,\mu_\xi)
+3\e\mcW_2^2(\mu^\n_\zeta,\mu_\zeta).
\end{align*}
Thus, an application of~\eqref{eq:W-empirical} gives the claim.
\end{proof}

\section*{Acknowledgments}
VF and JI gratefully acknowledge financial support of Sapere Aude Starting Grant 
8049-00021B ``Distributional Robustness in Assessment of Extreme Risk''.
JGC is grateful for the support of The Alan Turing Institute under  
EPSRC grant EP/N510129/1 and CoNaCyT scholarship 
2018-000009-01EXTF-00624 CVU699336.

\bibliographystyle{abbrv}
\bibliography{coupling}
\end{document}